\documentclass[12pt]{amsart}
\usepackage{graphicx}
\usepackage{latexsym}
\usepackage{fancyhdr}
\usepackage{amsmath, amssymb, amsthm}
\usepackage[all]{xy}
\usepackage{pdflscape}
\usepackage{longtable}
\usepackage{rotating}
\usepackage{verbatim}
\usepackage{hyperref}
\usepackage{subfigure}
\usepackage{mathrsfs}
\usepackage{tensor}
\usepackage{mdwlist}
\usepackage{etoolbox}
\usepackage{todonotes}
\usepackage{chngcntr}

\theoremstyle{plain}
\newtheorem*{lemma*}{Lemma}
\newtheorem{lemma}{Lemma}
\newtheorem*{theorem*}{Theorem}
\newtheorem{theorem}{Theorem}
\newtheorem*{proposition*}{Proposition}
\newtheorem{proposition}{Proposition}
\newtheorem*{corollary*}{Corollary}
\newtheorem{corollary}{Corollary}
\newtheorem*{claim*}{Claim}

\newtheorem{claim}{Claim}

\theoremstyle{definition}
\newtheorem*{assumption*}{Assumption}

\newtheorem*{definition*}{Definition}

\newtheorem*{convention*}{Convention}
\newtheorem*{example*}{Example}
\newtheorem{example}{Example}

\newtheorem*{algorithm*}{Algorithm}
\newtheorem*{remark*}{Remark}
\newtheorem{remark}{Remark}

\numberwithin{equation}{section}

\def\al{\alpha}
\def\be{\beta}

\def\de{\delta}

\def\et{\eta}
\def\th{\theta}

\def\si{\sigma}

\def\ta{\tau}

\def\ph{\phi}
\def\vh{\varphi}

\def\ps{\psi}
\def\om{\omega}

\def\De{\Delta}

\def\Si{\Sigma}

\def\Om{\Omega}

\def\C{\mathbb{C}}

\def\N{\mathbb{N}}

\def\R{\mathbb{R}}
\def\Z{\mathbb{Z}}

\def\cC{\mathcal{C}}
\def\cD{\mathcal{D}}

\def\cF{\mathcal{F}}

\def\cO{\mathcal{O}}

\def\sI{\mathscr{I}}

\def\p{\partial}

\def\<{\langle}
\def\>{\rangle}
\renewcommand{\o}{\circ}
\def\cq{{/\!\!/}}
\def\acts{\circlearrowleft}
\def\rep{(G\acts V,d,\si)}
\def\Lip{\on{Lip}}

\let\on=\operatorname
\newcommand{\sr}[1]%
{\ifmmode{}^\dagger\else${}^\dagger$\fi\ifvmode
\vbox to 0pt{\vss
 \hbox to 0pt{\hskip\hsize\hskip1em
 \vbox{\hsize3cm\raggedright\pretolerance10000
 \noindent #1\hfill}\hss}\vss}\else
 \vadjust{\vbox to0pt{\vss%
 \hbox to 0pt{\hskip\hsize\hskip1em%
 \vbox{\hsize3cm\raggedright\pretolerance10000%
 \noindent #1\hfill}\hss}\vss}}\fi%
}

\newcommand{\tC}[2]{\tensor[^{#1}]{C}{^{#2}}}

\title[Lifting differentiable curves from orbit spaces]
{Lifting differentiable curves from orbit spaces}

\author[Adam Parusi\'nski and Armin Rainer]
{Adam Parusi\'nski and Armin Rainer}

\address {Adam Parusi\'nski: Univ. Nice Sophia Antipolis, CNRS, LJAD, UMR 7351, 06108 Nice, France}

\email{adam.parusinski@unice.fr}

\address{Armin Rainer: Fakult\"at f\"ur Mathematik, Universit\"at Wien, 
Oskar-Morgenstern-Platz~1, A-1090 Wien, Austria}

\email{armin.rainer@univie.ac.at}

\begin{document}

\begin{abstract}
  Let $\rho : G \rightarrow \on{O}(V)$ be a real finite dimensional orthogonal representation of a compact Lie group, 
  let $\si = (\si_1,\ldots,\si_n) : V \to \R^n$, where $\si_1,\ldots,\si_n$ form a minimal system of homogeneous 
  generators of the 
  $G$-invariant polynomials on $V$, and set $d = \max_i \deg \si_i$.
  We prove that for each $C^{d-1,1}$-curve $c$ in $\si(V) \subseteq \R^n$ there exits a locally Lipschitz lift over $\si$, 
  i.e., a locally Lipschitz curve 
  $\overline c$ in $V$ so that $c = \si \o \overline c$, 
  and we obtain explicit bounds for the Lipschitz constant of $\overline c$ in terms of $c$. 
  Moreover, we show that each $C^d$-curve in $\si(V)$ admits a $C^1$-lift.
  For finite groups $G$ we deduce a multivariable version and some further results.
\end{abstract}

\thanks{Supported by the Austrian Science Fund (FWF), Grant P~26735-N25, and by ANR project STAAVF (ANR-2011 BS01 009).}
\keywords{Smooth mappings into orbit spaces, Lipschitz, $C^1$, and real analytic lifts}
\subjclass[2010]{22E45, 57S15, 14L24, 26A16}
\dedicatory{Dedicated to the memory of Mark Losik}
\date{May 10, 2015}

\maketitle

\section{Introduction and main results}

\subsection{Differentiable roots of hyperbolic polynomials}

Let us begin by describing the most important special case of our main theorem.

\begin{example}[Choosing differentiable roots of hyperbolic polynomials] \label{ex:1}
Let the symmetric group $\on{S}_n$ act on $\R^n$ by permuting the coordinates. The algebra of invariant polynomials 
$\R[\R^n]^{\on{S}_n}$ is generated by the elementary symmetric functions 
$\si_i = \sum_{j_1<\ldots<j_i} x_{j_1}\cdots x_{j_i}$. Considering the mapping 
$\si=(\si_1,\ldots,\si_n) : \R^n \to \R^n$, we may identify, in view of Vieta's formulas, 
each point $p$ of the image $\si(\R^n)$ uniquely with the 
monic polynomial $P_a = z^n + \sum_{j=1}^n a_j z^{n-j}$ whose unordered $n$-tuple of roots constitutes the fiber of $\si$ 
over $p$; two points in the fiber differ by a permutation. So the semialgebraic subset $\si(\R^n) \subseteq \R^n$ can be 
identified with the space of \emph{hyperbolic} polynomials of degree $n$, i.e., monic polynomials with all roots real. 

Suppose that the coefficients $a=(a_j)_{j=1}^n$ are functions depending in a smooth way on a real parameter $t$, i.e., 
$a : \R \to \R^n$ is a smooth curve with $a(\R) \subseteq \si(\R^n)$. Then we may ask how regular the roots of $P_a$ 
can be parameterized. This is a classical much studied problem with important applications in partial differential 
equations. We shall just mention three results which will be of interest in this paper.
\begin{enumerate}
        \item If $a$ is $C^{n-1,1}$ then any continuous parameterization of the roots of $P_a$ is locally Lipschitz with 
        uniform Lipschitz constant.
        \item If $a$ is $C^{n}$ then there exists a $C^1$-parameterization of the roots; actually any differentiable 
        parameterization is $C^1$.
        \item If $a$ is $C^{2n}$ then there exists a twice differentiable parameterization of the roots. 
\end{enumerate} 
The first result is a version of Bronshtein's theorem due to \cite{Bronshtein79}; a different proof was given 
by Wakabayashi \cite{Wakabayashi86}. In our recent note \cite{ParusinskiRainerHyp} we presented another independent proof 
of (1) the method of which works in the general situation considered in the present paper; see below. 
For the second and third  
result we refer to \cite{ColombiniOrruPernazza12}; see also \cite{ParusinskiRainerHyp} for a different proof, and  
\cite{Mandai85} and \cite{KLM04} for the same conclusions under stronger assumptions. 
The results (1), (2), and (3) are optimal. 
Most notably, there are $C^\infty$-curves $a$ so that the roots of $P_a$ 
do not admit a $C^{1,\om}$-parameterization for any modulus of continuity $\om$.      
\end{example}

Let $V$ be any finite dimensional Euclidean vector space.
For an open subset $U \subseteq \R^m$ and $p \in \N_{\ge 1}$, we denote by $C^{p-1,1}(U,V)$ 
the space of all mappings $f \in C^{p-1}(U,V)$ so that each partial derivative $\p^\al f$ of order
$|\al|=p-1$ is locally Lipschitz. It is a Fr\'echet space with the following system of seminorms,
\[
  \|f\|_{C^{p-1,1}(K,V)} = \|f\|_{C^{p-1}(K,V)} + \sup_{|\al|=p-1}\on{Lip}_K(\p^\al f), 
  \quad \on{Lip}_K(f) = \sup_{\substack{x,y \in K\\ x \ne y}} \frac{\|f(x)-f(y)\|}{\|x-y\|},
\]
where $K$ ranges over (a countable exhaustion of) the compact subsets of $U$; 
on $\R^m$ we consider the $2$-norm $\|~\|=\|~\|_2$. 
By Rademacher's theorem, the partial derivatives of order $p$ of a function $f \in C^{p-1,1}(U,V)$ exist almost 
everywhere.

\subsection{The general setup}

Let $G$ be a compact Lie group and let $\rho : G \rightarrow \on{O}(V)$ be an 
orthogonal representation in a real finite dimensional Euclidean 
vector space $V$ with inner product $\langle ~\mid~ \rangle$. 
For short we shall write $G \acts V$. 
By a classical theorem of Hilbert and Nagata, 
the algebra $\mathbb{R}[V]^{G}$ of invariant polynomials on $V$ 
is finitely generated. 
So let $\{\si_i\}_{i=1}^n$ be a system of homogeneous generators 
of $\mathbb{R}[V]^{G}$ which we shall also call a system of \emph{basic invariants}. 

A system of basic invariants $\{\si_i\}_{i=1}^n$
is called \emph{minimal} if there is no polynomial relation of the 
form $\si_i = P(\si_1,\ldots,\widehat{\si_i},\ldots,\si_n)$, or equivalently, 
$\{\si_i\}_{i=1}^n$ induces a basis of the real vector space $\R[V]^G_+/(\R[V]^G_+)^2$, 
where $\R[V]^G_+ = \{f \in \R[V]^G : f(0)=0\}$; cf.\ \cite[Section 3.6]{DK02}.   
The elements in a minimal system of basic invariants may not be unique but its number and its degrees 
$d_i := \deg \si_i$ are unique. Let us set
\[
  d:= \max_{i=1,\ldots,n} d_i.
\]

Given a system of basic invariants $\{\si_i\}_{i=1}^n$, 
we consider the \emph{orbit mapping} 
$\sigma = (\sigma_1,\ldots,\sigma_n) : V \rightarrow \mathbb{R}^n$. 
The image $\sigma(V)$ is a semialgebraic set in the categorical quotient 
$V /\!\!/ G := \{y \in \mathbb{R}^n : P(y) = 0 ~\mbox{for all}~ P \in \sI\}$,  
where $\sI$ is the ideal of relations between $\sigma_1,\ldots,\sigma_n$. 
Since $G$ is compact, $\sigma$ is proper and separates orbits of $G$, and 
it thus induces a homeomorphism $\tilde \si$ between the orbit space $V/G$ and $\sigma(V)$.

Let $H = G_v = \{g \in G : gv=v\}$ be the isotropy group of $v \in V$ and $(H)$ its conjugacy class 
in $G$;  $(H)$ is called the \emph{type} of the orbit $Gv=\{gv : g \in G\}$. Let $V_{(H)}$ be the union of all
orbits of type $(H)$. Then $V_{(H)}/G$ is a smooth manifold and  
the collection of connected components of the manifolds $V_{(H)}/G$ 
forms a stratification of $V/G$ by orbit type; cf.\ \cite{Schwarz80}.  
Due to \cite{Bierstone75}, $\tilde \si$ is an isomorphism between the orbit type stratification of $V/G$ 
and the natural stratification of $\si(V)$ as a semialgebraic set; it is analytically locally trivial
and thus satisfies Whitney's conditions (A) and (B). 
The inclusion relation on the set of subgroups of $G$ induces a partial ordering on the family of orbit types. 
There is a unique minimal orbit type, the principal orbit type, corresponding 
to the open and dense submanifold $V_{\on{reg}}$ 
consisting of points $v$, where the slice 
representation $G_v \acts N_v$ is trivial; see Subsection~\ref{ssec:slice} below.
The projection $V_{\on{reg}} \to V_{\on{reg}}/G$ is a locally trivial fiber bundle.
There are only finitely many isomorphism classes of slice representations.

A representation $G \acts V$ is called \emph{polar}, 
if there exists a linear subspace $\Si \subseteq V$, 
called a \emph{section}, 
which meets each orbit orthogonally; cf.\ \cite{Dadok85}, \cite{DK85}. 
The trace of the $G$-action on $\Si$ is the action of the \emph{generalized Weyl group} 
$W(\Si) = N_G(\Si)/Z_G(\Si)$ on $\Si$, where 
$N_G(\Si) := \{g \in G : g\Si = \Si\}$ and 
$Z_G(\Si) := \{g \in G : gs = s \text{ for all } s \in \Si\}$. 
This group is finite, and it is a reflection group if 
$G$ is connected. The algebras $\R[V]^G$ and $\R[\Si]^{W(\Si)}$ are isomorphic via restriction, 
by a generalization of Chevalley's restriction theorem due to \cite{DK85} and  
independently \cite{Terng85}, 
and thus the orbit spaces $V/G$ and $\Si/W(\Si)$ are isomorphic.

We shall fix a minimal system of basic invariants $\{\si_i\}_{i=1}^n$ and 
the corresponding orbit mapping $\si$. The given data will be abbreviated by the 
tuple $\rep$.  

\subsection{Smooth structures on orbit spaces}
We review some ways to endow the orbit space $V/G$ with a smooth structure and stress the connection to the lifting 
problem studied in this paper. The results and constructions mentioned in this subsection will not be used later in the paper.

A smooth structure on a non-empty set $X$ can be introduced by specifying any of the 
following families of mappings 
together with some compatibility conditions:
\begin{itemize}
  \item the smooth functions on $X$ (differential space) 
  \item the smooth mappings into $X$ (diffeological space)
  \item the smooth curves in $X$ and the smooth functions on $X$ (Fr\"olicher space)
\end{itemize}
More precisely: A \emph{differential structure} on $X$ is a family $\cF_X$ of functions $X \to \R$, along with the associated 
initial topology on $X$, so that 
\begin{itemize}
   \item if $f_1,\ldots,f_n \in \cF_X$ and $g \in C^\infty(\R^n)$ then $g \o (f_1,\ldots,f_n) \in \cF_X$
   \item if $f : X \to \R$ is locally the restriction of a function in $\cF_X$ then $f \in \cF_X$.   
 \end{itemize}
The pair $(X,\cF_X)$ is called a \emph{differential space}. 

A \emph{diffeology} on $X$ is a family $\cD_X$ of mappings $U \to X$, where $U$ is any \emph{domain}, 
i.e., open in some $\R^n$, so that 
\begin{itemize}
   \item $\cD_X$ contains all constant mappings $\R^n \to X$ (for all $n$)
   \item for each $p : U \to X \in \cD_X$, each domain $V$, and each $q \in C^\infty(V,U)$, also $p\o q \in \cD_X$
   \item if $p : U \to X$ is locally in $\cD_X$ then $p \in \cD_X$.
\end{itemize} 
The pair $(X,\cD_X)$ is called a \emph{diffeological space}.

A \emph{Fr\"olicher structure} on $X$ is a pair $(\cC_X,\cF_X)$ consisting of a subset $\cC_X \subseteq X^\R$ 
and a subset $\cF_X \subseteq \R^X$ so that
\begin{itemize}
  \item $f \in \cF_X$ if and only if $f \o c \in C^\infty(\R,\R)$ for all $c \in \cC_X$
  \item $c \in \cC_X$ if and only if $f \o c \in C^\infty(\R,\R)$ for all $f \in \cF_X$.  
\end{itemize}
The triple $(X,\cC_X,\cF_X)$ is called a \emph{Fr\"olicher space}. 
The Fr\"olicher structure on $X$ generated by a subset $\cC \subseteq X^\R$ (respectively $\cF \subseteq \R^X$) is the 
finest (respectively coarsest) Fr\"olicher structure $(\cC_X,\cF_X)$ on $X$ with $\cC \subseteq \cC_X$ 
(respectively $\cF \subseteq \cF_X$).

A mapping $\ph : X \to Y$ between two spaces of the same kind is called \emph{smooth} if
\begin{itemize}
  \item $\ph^* \cF_Y \subseteq \cF_X$ in the case of differential spaces
  \item $\ph_* \cD_X \subseteq \cD_Y$ in the case of diffeological spaces
  \item $\ph_* \cC_X \subseteq \cC_Y$, equivalently $\ph^* \cF_Y \subseteq \cF_X$, 
  equivalently $\cF_Y \o \ph \o \cC_X \in C^\infty$ in the case of Fr\"olicher spaces.
\end{itemize}

Any of the above forms a category, and 
the category of smooth finite dimensional manifolds with smooth mappings in the usual sense forms a full 
subcategory in each of them. 

\medskip

The orbit space $V/G$ can be given a differential structure by defining a function on $V/G$ to be smooth 
if its composite with the projection $V \to V/G$ is smooth, i.e., $\cF_{V/G}= C^\infty(V/G) \cong C^\infty(V)^G$. 
On the other hand $\si(V)$ has a differential  
structure defined by restriction of the smooth functions on $\R^n$, i.e., 
$\cF_{\si(V)} = \{f|_{\si(V)} : f \in C^\infty(\R^n)\}$. 
By Schwarz' theorem \cite{Schwarz75}, $\si^* C^\infty(\R^n) = C^\infty(V)^G$ and so $\tilde \si$ is an isomorphism of 
$V/G$ and $\sigma(V)$ together 
with their differential structures. 
In other words quotient and subspace differential structure coincide.
We have 
\begin{align*}
  C^\infty(\R,\si(V)) &:= \{c \in C^\infty(\R,\R^n) : c(\R) \subseteq \si(V)\} \\
  &= \{c \in \si(V)^\R : f\o c \in C^\infty(\R,\R) \text{ for all } f \in C^\infty(V)^G\}.  
\end{align*}
We may also consider the curves in $\si(V)$ that admit a smooth lift over $\si$,
\[
  \si_* C^\infty(\R,V) = \{\si \o c : c \in C^\infty(\R,V)\}.
\]
In general the inclusion $\si_* C^\infty(\R,V) \subseteq C^\infty(\R,\si(V))$ is strict (cf.\ Example \ref{ex:1}). 
The set of functions $C^\infty(V)^G$ on the one hand and the set of curves $\si_* C^\infty(\R,V)$ on the 
other hand give rise to Fr\"olicher space structures on the orbit space $V/G = \si(V)$ that turn out 
to coincide:  
The Fr\"olicher structure on $\si(V)$ generated by $C^\infty(V)^G$ as well as that generated by 
$\si_* C^\infty(\R,V)$ is $(C^\infty(\R,\si(V)),C^\infty(V)^G)$. Indeed, we have 
\[
  C^\infty(V)^G \cong \{f \in \R^{\si(V)} : f \o c \in C^\infty(\R,\R) \text{ for all } c \in \si_* C^\infty(\R,V)\},
\]
for if $f \o c \in C^\infty$ for all $c \in \si_* C^\infty(\R,V)$ then $f \o \si$ is $C^\infty$, 
by Boman's theorem \cite{Boman67}. 
It follows that the quotient and the subspace Fr\"olicher structure coincide on $\si(V)$. 

However, the quotient diffeology $\cD_q$ and the subspace diffeology $\cD_s$ on $\si(V)$ fall apart.  
The quotient diffeology $\cD_q$ with respect to 
the orbit mapping $\si: V \to \si(V)$ is the finest diffeology of $\si(V)$ such that $\si: V \to \si(V)$ is smooth.  
A mapping $f : U \to \si(V)$ belongs to $\cD_q$ if and only if it lifts locally over $\si$, 
i.e., for each $x \in U$ there is 
a neighborhood $U_0$ and a $C^\infty$-mapping $\overline f : U_0 \to V$ so that $f = \si \o \overline f$ on $U_0$.
The subspace diffeology $\cD_s$ on $\si(V)$ is the coarsest diffeology    
of $\si(V)$ such that the inclusion $\si(V) \hookrightarrow \R^n$ is smooth.
A mapping $U \to \si(V)$ belongs to $\cD_s$ if and only if the 
composite $U \to \si(V) \hookrightarrow \R^n$ is smooth. 
Evidently, $\cD_q \subseteq \cD_s$, and the inclusion is strict (cf.\ Example \ref{ex:1}). 

\subsection*{The orbit space as a differentiable space} 

Let us finally consider $V/G$ as a differentiable space in the sense of Spallek \cite{Spallek69}. 
We follow the presentation in \cite{Navarro-GonzalezSancho-de-Salas03}.

An $\R$-algebra $A$ is called a \emph{differentiable algebra} if it is isomorphic to $C^\infty(\R^n)/\mathfrak{a}$ 
for some positive integer $n$ and some closed ideal $\mathfrak a$ in $C^\infty(\R^n)$.
Any differentiable algebra $A$ has a unique Fr\'echet topology such that the algebra isomorphism 
$A \cong C^\infty(\R^n)/\mathfrak{a}$ is a homeomorphism, cf.\ \cite[Theorem 2.23]{Navarro-GonzalezSancho-de-Salas03}.
The real spectrum $\on{Spec}_r A$ of $A = C^\infty(\R^n)/\mathfrak{a}$ is homeomorphic to $\{x \in \R^n : f(x) = 0,~ \forall f \in \mathfrak a\}$, 
cf.\ \cite[Proposition 2.13]{Navarro-GonzalezSancho-de-Salas03}.

A locally ringed space $(X,\cO_X)$ is said to be an \emph{affine differentiable space} if it is 
isomorphic to the real spectrum $(\on{Spec}_r A, \tilde A)$ of some differential algebra $A$. Here $\tilde A$ is the sheaf 
associated to the presheaf $U \leadsto A_U$, where $A_U = \{a/b : a,b \in A,~b(x) \ne 0 , ~\forall x \in U\}$ denotes the localization. 
A locally ringed space $(X,\cO_X)$ is said to be a \emph{differentiable space} if 
each point $x \in X$ has an open neighborhood $U$ in $X$ such that $(U,\cO_X|_U)$ is an affine differentiable space.
Sections of $\cO_X$ on an open set $U \subseteq X$ are called \emph{differentiable functions} on $U$.
A differentiable space $(X,\cO_X)$ is said to be \emph{reduced} if for each open set $U \subseteq X$ and every 
differentiable function $f \in \cO_X(U)$, we have $f=0$ if and only if $f(x) =0$ for all $x \in U$. 

The space $\R^n$ is a reduced affine differentiable space: let $C^\infty_{\R^n}$ denote the sheaf of $C^\infty$-functions on $\R^n$, 
then
$(\on{Spec}_r C^\infty(\R^n), C^\infty_{\R^n})  \cong (\R^n,C^\infty_{\R^n})$, cf.\ \cite[Example 3.15]{Navarro-GonzalezSancho-de-Salas03}.

Let $Z$ be a topological subspace of $\R^n$. 
A continuous function $f : Z \to \R$ is said to be of class $C^\infty$ if each point $z \in Z$ has an open neighborhood  
$U_z$ in $\R^n$ and there exists $F \in C^\infty(U_z)$ such that $f|_{Z \cap U_z} = F|_Z$. 
Thus we obtain a sheaf $C^\infty_Z$ of continuous 
functions on $Z$, and $(Z,C^\infty_Z)$ is a reduced affine differentiable space; cf.\ \cite[Corollary 5.8]{Navarro-GonzalezSancho-de-Salas03}.
The category of reduced differentiable spaces is equivalent to the category of reduced ringed spaces $(X,\cO_X)$ with the property that 
each $x \in X$ has an open neighborhood $U$ such that $(U,\cO_X|_U)$ is isomorphic to $(Z, C^\infty_Z)$ for some closed 
subset $Z$ of an affine space $\R^n$; cf.\ \cite[Theorem 3.23]{Navarro-GonzalezSancho-de-Salas03}.

\medskip

Let us turn to our situation. We equip the orbit space $V/G$ (with the quotient topology and) with the structural sheaf 
$\cO_{V/G}$, where $\cO_{V/G}(U) := \{f \in C^0(U,\R) : f \o \pi \in C^\infty(\pi^{-1}(U))\} \cong C^\infty(\pi^{-1}(U))^G$ and $\pi : V \to V/G$ denotes the quotient mapping. 
On the closed subset $\si(V)$ of $\R^n$ we consider the structure of reduced affine differentiable space induced by $\R^n$, 
i.e., $(\si(V), C^\infty_{\si(V)})$. It follows from Schwarz's theorem and the localization theorem for smooth functions 
(see \cite[p.\ 28]{Navarro-GonzalezSancho-de-Salas03}) that $\si$ induces an isomorphism of the differentiable 
spaces $(V/G,\cO_{V/G})$ and $(\si(V), C^\infty_{\si(V)})$; see \cite[Theorem 11.14]{Navarro-GonzalezSancho-de-Salas03}.  
Note that the reduced affine differentiable space $(V/G,\cO_{V/G})$ is the differential space $(V/G,\cF_{V/G})$ considered above.

\subsection{The main results}

In this paper we shall be concerned with the lifting properties of arbitrary elements in $C^\infty(\R,\si(V))$ 
(or in $\cD_s$).

Let $I \subseteq \R$ be an open interval and
let $c : I \to V/G = \si(V) \subseteq \mathbb{R}^n$ be a curve 
in the orbit space $V/G$ of $\rep$.
A curve $\overline c : I \to V$ is called a \emph{lift of $c$ over $\si$}, if $c = \sigma \circ \overline c$ 
holds. 
We will consider curves $c$ in $V/G = \si(V)$ that are in some H\"older class $C^{k,\al}$, this means 
that $c$ is $C^{k,\al}$ as curve in $\R^n$ with the image contained in $\si(V)$, and it will be denoted by 
$c \in C^{k,\al}(I,\si(V))$. 
Note that any $c \in C^0(I,\si(V))$ admits a lift $\overline c \in C^0(I,V)$, by \cite{MY57} 
or \cite[Proposition~3.1]{KLMR05}.
The problem of lifting curves over invariants is independent of the 
choice of a system of basic invariants as any two such choices differ by a polynomial 
diffeomorphism.

This problem was considered in this generality for the first time in \cite{AKLM00}; it was shown that 
$\si_* C^\infty(\R,V)$ contains all elements in $C^\infty(\R,\si(V))$ that do not meet lower dimensional strata of 
$\si(V)$ with infinite order of flatness. 
A $C^d$-curve in $\si(V)$ admits a differentiable lift, due to \cite{KLMR05}. 
In \cite{KLMR06} and \cite{KLMRadd} 
the following generalization of Example~\ref{ex:1} was obtained:
Let $G$ be \emph{finite}, write $V=V_1 \oplus \cdots \oplus V_l$ as an orthogonal 
direct sum of irreducible subspaces $V_i$, and set $$k = \max \{d,k_1,\ldots,k_l\},$$ where $k_i$ 
is the minimal cardinality of non-zero orbits in $V_i$.
Then $C^k$ (resp.\ $C^{k+d}$) curves in $V/G$ admit $C^1$ (resp.\ twice differentiable) lifts. 
This result was achieved by reducing the general case $G \acts V$ to the case of the standard action of the symmetric 
group $\on{S}_n \acts \R^n$ and then applying Bronshtein's theorem. This technique works only for finite groups and 
it yields a corresponding result for polar representations (since the associated Weyl group is finite).

The ideas of our new proof of Bronshtein's theorem in \cite{ParusinskiRainerHyp} 
led us to the main results of this paper: 
\begin{itemize}
  \item We show that $C^{d-1,1}$-curves in the orbit space of \emph{any} representation $\rep$ admit $C^{0,1}$-lifts 
  and we obtain explicit bounds for the Lipschitz constants (Theorem~\ref{thm:main}).
  \item We prove that $C^{d}$-curves in the orbit space of \emph{any} representation $\rep$ admit $C^{1}$-lifts 
  (Theorem~\ref{mainC1}). 
  \item If $G$ is a finite group we find that 
  \begin{itemize}
    \item each continuous lift of a $C^{d-1,1}$-curve is $C^{0,1}$ (Corollary~\ref{cor:fin}),
    \item each differentiable lift of a $C^{d}$-curve is $C^1$ (Corollary~\ref{cor:twice}),
    \item each $C^{2d}$-curve admits a twice differentiable lift (Corollary~\ref{cor:twice}).
  \end{itemize}
  \item If $G$ is a finite group we also obtain that each continuous lift of a $C^{d-1,1}$-
  mapping of \emph{several variables} into the orbit space is $C^{0,1}$ with uniform Lipschitz constants
  (Corollary~\ref{cor:several}).
  \item As a by-product of the problem of gluing together local lifts (see Section \ref{sec:proofC1}) we show that 
  real analytic curves in the orbit space of \emph{any} representation $\rep$ can be lifted \emph{globally} 
  (Theorem~\ref{thm:realanalytic}). This extends a result of \cite{AKLM00} who proved the existence of local real 
  analytic lifts, and global ones if $G \acts V$ is polar.     
\end{itemize}
Our proofs do not rely on Bronshtein's result but we reprove it.  

\begin{theorem} \label{thm:main}
  Let $\rep$ be a real finite dimensional orthogonal representation of a compact Lie group. 
  Then any $c \in C^{d-1,1}(I,\si(V))$ admits a lift $\overline c \in C^{0,1}(I,V)$. More precisely, 
  for any relatively compact subset $I_0 \Subset I$, there is a neighborhood $I_1$ with $I_0 \Subset I_1 \Subset I$ 
  so that 
  \begin{align}\label{eq:uniform}
  \on{Lip}_{I_0}(\overline c)  & \le C\, \big(\max_i \|c_i\|^{\frac 1 {d_i}}_{C^{d-1,1}(\overline I_1)}\big) 
      \\
  \notag  
     &  \le \tilde C\, \big(1+\max_i \|c_i\|_{C^{d-1,1}(\overline I_1)}\big) 
  \end{align}
for constants $C$ and $\tilde C$ depending only on the intervals $I_0,I_1$ and 
on the isomorphism classes of the slice representations of $G \acts V$ 
and respective minimal systems of basic invariants. 
(More precise bounds are stated in Subsection \ref{ssec:bounds}.)
\end{theorem}

\begin{remark} \label{rem:main}
  The statement of Theorem \ref{thm:main} reads ``there is a $C^{0,1}$-lift $\overline c$ on the whole interval $I$ so that for all 
      $I_0 \Subset I$ there is a neighborhood $I_1$ such that \eqref{eq:uniform} holds''. 
  Our proof also yields ``for all intervals $I_0$ and $I_1$ with $I_0 \Subset I_1 \Subset I$ there is a Lipschitz lift $\overline c$ 
  on $I_0$ satisfying \eqref{eq:uniform}''.      
\end{remark}

\begin{convention*}
  We will denote by $C=C(G \acts V,\ldots)$ \emph{any} constant depending only on $G \acts V, \ldots$; its 
  value may vary from line to line.  
  Specific constants will bear a subscript like $C_0=C_0(\ldots)$ or $C_1=C_1(\ldots)$.   
  The dependence on $G \acts V$ is to be understood in the following way. 
  For every isomorphism class $H \acts W$ of slice representations of $G \acts V$ fix a minimal system of 
  basic invariants; 
  note that there are only finitely many slice representations up to isomorphism and that 
  $G\acts V$ coincides with its slice representation at $0$. 
  Writing $C=C(G\acts V)$ we mean that the constant $C$ only depends on the isomorphism classes of the 
  slice representations 
  of $G\acts V$ and on the respective fixed minimal systems of basic invariants.     
\end{convention*}

Our second main result is the following.

\begin{theorem} \label{mainC1}
  Let $\rep$ be a real finite dimensional orthogonal representation of a compact Lie group. 
  Then any $c \in C^{d}(I,\si(V))$ admits a lift $\overline c \in C^{1}(I,V)$.
\end{theorem}

Theorem~\ref{thm:main} and Theorem~\ref{mainC1} will be proved in Section~\ref{sec:proof} and Section~\ref{sec:proofC1}, 
respectively.

\medskip

For finite groups $G$ we can show more:

\begin{corollary} \label{cor:fin}
  Let $\rep$ be a real finite dimensional orthogonal representation of a finite group. 
  Then any continuous lift $\overline c$ of $c \in C^{d-1,1}(I,\si(V))$ is locally Lipschitz and satisfies 
  \eqref{eq:uniform} for all intervals $I_0 \Subset I_1 \Subset I$.
\end{corollary}

\begin{proof}
  Let $\tilde c$ be any continuous lift of $c$, and let $I_0 \Subset I_1 \Subset I$.
  Let $\overline c$ be the Lipschitz lift on $I_0$ 
  provided by Remark~\ref{rem:main}.
  Let $s,t \in I_0$, $s<t$. 
  For each $g \in G$ consider the closed subset $J_g := \{r \in [s,t] : \tilde c(r) = g \overline c(r)\}$ of $[s,t]$. 
  As $[s,t] = \cup_{g \in G} J_g$ there exists a subset $\{g_1,\ldots,g_\ell\} \subseteq G$ and 
  finite sequence $s=t_0 < t_1 < \cdots < t_\ell = t$ so that $t_{i-1},t_i \in J_{g_i}$ for all $i=1,\ldots,\ell$. 
  Then
  \[
    \|\tilde c(s) - \tilde c(t)\| \le  \sum_{i=1}^\ell \|g_i \overline c(t_{i-1}) - g_i\overline c(t_i)\| \le 
    \on{Lip}_{I_0}(\bar c)\, (t-s), 
  \] 
  which implies the assertion.
\end{proof}

Corollary~\ref{cor:fin} readily implies the following result on lifting of mappings in several variables.

\begin{corollary} \label{cor:several}
  Let $\rep$ be a real finite dimensional orthogonal representation of a finite group. 
  Let $U \subseteq \R^m$ be open and let $f \in C^{d-1,1}(U,\si(V))$. 
  Then any continuous lift $\overline f : U \supseteq \Om \to V$ of $f$, on an open subset $\Om$ of $U$, 
  is locally Lipschitz.
  More precisely, 
  for any pair of relatively compact open subsets $\Om_0 \Subset \Om_1 \Subset \Om$ we have   
  \begin{align}\label{eq:uniform2}
  \on{Lip}_{\Om_0}(\overline f)  & \le C\, \big(\max_i \|f_i\|^{\frac 1 {d_i}}_{C^{d-1,1}(\overline \Om_1)}\big) 
      \\
  \notag  
     &  \le \tilde C\, \big(1+\max_i \|f_i\|_{C^{d-1,1}(\overline \Om_1)}\big), 
  \end{align}
  for constants $C=C(G\acts V,\Om_0,\Om_1,m)$ and $\tilde C=\tilde C(G\acts V,\Om_0,\Om_1,m)$.
\end{corollary}

\begin{remark*}\hfill
  \begin{enumerate}
    \item If $G$ has positive dimension and $\overline f$ is a $C^{0,1}$-lift of $f$, we may obtain 
    a continuous lift of $f$ that is not locally Lipschitz by simply multiplying $\overline f$ by a suitable continuous 
    mapping $g : U \to G$.
    \item In general there are representations and smooth mappings into the orbit space of such 
    which do not admit continuous lifts. For instance, the orbit space of a finite rotation group of $\R^2$ 
    is homeomorphic to the set $C$ obtained from the sector 
    $\{r e^{i\vh} \in \C : r \in [0,\infty), 0 \le \vh \le \vh_0\}$ 
    by identifying the rays that constitute its boundary. 
    A loop on $C$ cannot be lifted to a loop in $\R^2$ unless it is homotopically trivial in $C \setminus \{0\}$.
  \end{enumerate}
\end{remark*}

\begin{proof} 
  Let $\overline f : U \supseteq \Om \to V$ be a continuous lift of $f$ on $\Om$.
  Without loss of generality we may assume that $\Om_0$ and $\Om_1$ are open boxes parallel to the 
  coordinate axes, $\Om_i = \prod_{j=1}^m I_{i,j}$, $i=0,1$, with $I_{0,j} \Subset I_{1,j}$ for all $j$. 
  Let $x,y \in \Om_0$ and set $h:= y-x$.
  Let $\{e_i\}_{i=1}^m$ denote the standard unit vectors in $\R^m$.
  For any $z$ in the orthogonal projection of $\Om_0$ on the hyperplane $x_j =0$ 
  consider the curve $\overline f_{z,j} : I_{0,j} \to V$ defined by 
  $\overline f_{z,j}(t) := \overline f(z + t e_j)$. 
  By Corollary~\ref{cor:fin}, each $\overline f_{z,j}$ is Lipschitz and $C := \sup_{z,j} \on{Lip}_{I_{0,j}}(\overline f_{z,j}) < \infty$.
  Thus 
  \[
    \|\overline f(x)-\overline f(y)\| 
    \le \sum_{j=0}^{m-1} 
    \Big\|\overline f\big(x+\sum_{k=1}^j h_k e_k\big) - \overline f\big(x+\sum_{k=1}^{j+1} h_k e_k\big)\Big\| 
    \le C \|h\|_1 \le C \sqrt{m} \|h\|_2.
  \] 
  The bounds \eqref{eq:uniform2} follow from \eqref{eq:uniform}.
\end{proof}

\begin{corollary} \label{cor:twice}
  Let $(G \acts V,d,\si)$ be a real finite dimensional orthogonal representation of a finite group. Then: 
  \begin{enumerate}
     \item Any differentiable lift of $c \in C^{d}(I,\si(V))$ is $C^1$.
     \item Any $c \in C^{2d}(I,\si(V))$ admits a twice differentiable lift.
   \end{enumerate} 
\end{corollary}

\begin{proof}
  This follows from Corollary~\ref{cor:fin}. 
  It can be proved as in \cite{KLMR06}; see also \cite{KLMRadd}. 
\end{proof}

\subsection{Further examples}

\begin{example}[Choosing differentiable eigenvalues of real symmetric matrices]
Let the orthogonal group $\on{O}(n) = \on{O}(\R^n)$ act by conjugation on the real vector space $\on{Sym}(n)$ of real 
symmetric $n \times n$ matrices, 
$\on{O}(n) \times \on{Sym}(n) \ni (S,A) \mapsto SAS^{-1}=SAS^t \in \on{Sym}(n)$. The algebra of invariant polynomials 
$\R[\on{Sym}(n)]^{\on{O}(n)}$ is isomorphic to $\R[\on{Diag(n)}]^{\on{S}_n}$ by restriction, 
where $\on{Diag}(n)$ is the vector space 
of real diagonal $n \times n$ matrices upon which $\on{S}_n$ acts by permuting the diagonal entries. 
More precisely, $\R[\on{Sym}(n)]^{\on{O}(n)} = \R[\Si_1,\ldots,\Si_n]$, where
$\Si_i(A) = \on{Trace}(\bigwedge^i A : \bigwedge^i \R^n \to \bigwedge^i \R^n)$ is the $i$th
characteristic coefficient of $A$ and $\Si_i|_{\on{Diag}(n)} = \si_i$, where $\si_i$ is the $i$th 
elementary symmetric polynomial and we identify $\on{Diag}(n) \cong \R^n$ (cf.\ \cite[7.1]{MichorH}). 
This means that 
the representation $\on{O}(n) \acts \on{Sym}(n)$ is polar and $\on{Diag(n)}$ forms a section.  

A smooth curve $A : \R \to \on{Sym}(n)$ of symmetric matrices induces a smooth curve of hyperbolic polynomials $P_A$ 
(the characteristic polynomial of $A$), i.e., a smooth curve in the semialgebraic set 
$\si(\on{Diag}(n)) \cong \si(\R^n)$ 
from Example~\ref{ex:1}. Then (1), (2), and (3) in Example~\ref{ex:1} imply regularity results for the eigenvalues of 
$t \mapsto A(t)$ which however turn out to be not optimal. In fact we have the following optimal results. 
\begin{enumerate}
        \item If $A$ is $C^{0,1}$ then any continuous parameterization of the eigenvalues of $A$ is locally Lipschitz 
        with uniform Lipschitz constant.
        \item If $A$ is $C^1$ then there exists a $C^1$-parameterization of the eigenvalues; actually any differentiable 
        parameterization is $C^1$.
        \item If $A$ is $C^{2}$ then there exists a twice differentiable parameterization of the eigenvalues. 
\end{enumerate}
The first result follows from a result due to Weyl \cite{Weyl12}, the second and third were shown in \cite{RainerN}. 
Actually, these results 
are true for normal complex matrices and, in appropriate form, even for normal operators in Hilbert space with common 
domain of definition and compact resolvents; see \cite{RainerN}.  

Here the curve $P_A$ in the orbit space is the projection of the curve $A$ under $\on{Sym}(n) \to \on{Sym}(n)/\on{O}(n)$
and is then lifted over $\on{Diag}(n) \to \on{Diag}(n)/\on{S}_n$. 
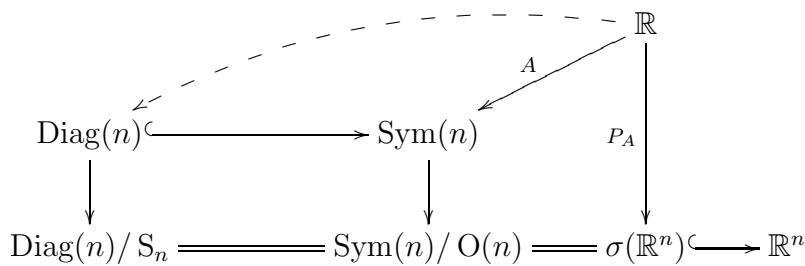
\[
  \xymatrix{
     && & \R \ar[dl]_{A} \ar[dd]_{P_A} \ar@/_20pt/@{-->}[dlll] & \\
    \on{Diag}(n) \ar@{^{(}->}[rr] \ar[d] && \on{Sym}(n) \ar[d] &&   \\
    \on{Diag}(n)/\on{S}_n \ar@{=}[rr] && \on{Sym}(n)/\on{O}(n) \ar@{=}[r] & \si(\R^n) \ar@{^{(}->}[r] & \R^n   
  } 
\]
\medskip
\end{example}

\begin{example}[Decomposing nonnegative functions into differentiable sums of squares]
Let the orthogonal group $\on{O}(n)$ act in the standard way on $\R^n$. Then the algebra of invariant polynomials 
$\R[\R^n]^{\on{O}(n)}$ is generated by $\si = \sum_{i=1}^n x_i^2$. The orbit space $\R^n/\on{O}(n)$ can be 
identified with the half-line $\R_{\ge 0} = [0,\infty) = \si(\R^n)$. Each line through the origin of $\R^n$ forms a
section of $\on{O}(n) \acts \R^n$.

Given a smooth nonnegative function $f$, decomposing $f$ into sums of squares amounts to lifting $f$ over $\si$. 
Applying Example~\ref{ex:1}\thetag{1} (actually its multiparameter analogue which follows easily; 
see Corollary~\ref{cor:several})
implies that:
\begin{enumerate}
   \item Any nonnegative $C^{1,1}$ function $f : \R^m \to \R$ is the square of a $C^{0,1}$ function.
\end{enumerate} 
The image of this lift lies in a section of $\on{O}(n) \acts \R^n$. 
This does not apply to the solutions in the
following stronger results which benefit from the additionally 
available space.
\begin{enumerate}
    \item[(2)] Any nonnegative $C^{3,1}$ function $f : \R^m \to \R$ is a sum of $n=n(m)$ squares of $C^{1,1}$ functions.
    \item[(3)] Let $p \in \N$. Any nonnegative $C^{2p}$ function $f: \R \to \R$ is the sum of two squares of $C^p$ functions. 
\end{enumerate} 
Result \thetag{2} was stated by Fefferman and Phong while proving their celebrated inequality in \cite{FP78}; see also 
\cite[Lemma 4]{Guan97}. 
This is sharp in the sense that there exist $C^\infty$ functions $f : \R^m \to \R$, for $m\ge 4$, 
that are not sums of squares of $C^2$ functions; see \cite{BBCP06}. Result \thetag{3} is due to \cite{Bony05}; 
the decomposition depends on $p$. 
\end{example}

\section{Reduction to slice representations}

Let $(G \acts V,d,\si)$ be fixed. Let $V^G = \{v \in V : Gv=v\}$ be the linear subspace of invariant vectors.

\subsection{Dominant invariant}
We may assume without loss of generality that 
\begin{equation} \label{eq:red1}
  \text{$\si_1(v) = \< v \mid v\> = \|v\|^2$ for all $v \in V$.}  
\end{equation}
Indeed, if the invariant polynomial $v \mapsto \< v \mid v\>$ does not belong to the minimal 
system of basic invariants, we just add it. This does not change $d$ unless $d = 1$. 
But in the latter case $V=V^G$ and there is nothing to prove. In fact, if $d=1$ then 
the elements in a minimal system of basic invariants form a system of linear coordinates on $V$.

Under the assumption \eqref{eq:red1} the invariant $\si_1$ is dominant in the following sense:
for all $j =1,\ldots,n$ and all $v \in V$, 
\begin{equation} \label{eq:dom}
  |\si_j(v)|^{\frac 1 {d_j}} \le C\, |\si_1(v)|^{\frac 1 {d_1}} = C\, \|v\|,
\end{equation}
where $C=C(\si)$. Indeed, $|\si_j(v)| \le \max_{\|w\|=1} |\si_j(w)| \, \|v\|^{d_j}$, by homogeneity.

\subsection{Removing fixed points} \label{ssec:fixedpoints}

Let $V'$ be the orthogonal complement of $V^G$ in $V$. 
Then we have $V = V^G \oplus V'$, 
$\mathbb{R}[V]^G = \mathbb{R}[V^G] \otimes \mathbb{R}[V']^G$ and 
$V/G = V^G \times V'/G$. The following lemma is obvious. 

\begin{lemma} \label{fix}
Any lift $\overline c$ of a curve $c=(c_0,c_1)$ in $V^G \times V'/G$ has the form 
$\overline c = (c_0,\overline c_1)$, 
where $\overline c_1$ is a lift of $c_1$. 
\end{lemma}

In view of Lemma~\ref{fix} we may assume that 
\begin{equation} \label{eq:red2}
  V^G = \{0\}.
\end{equation}

\subsection{The slice theorem}  \label{ssec:slice}
For a point $v \in V$ we denote by  
$N_v = T_v(Gv)^{\bot}$ the normal subspace of the orbit $Gv$ at $v$.
It carries a natural $G_v$-action $G_v \acts N_v$.  
The crossed product (or associated bundle) $G \times_{G_v} N_v$ carries the structure 
of an affine real algebraic variety as the categorical (and geometrical) quotient 
$(G \times N_v)\cq G_v$ with respect to the action $G_v \acts (G \times N_v)$ given by 
$h (g,x) = (gh^{-1},hx)$. 
Denote by $[g,x]$ the element of $G \times_{G_v} N_v$ represented by $(g,x) \in G \times N_v$. 
The $G$-equivariant polynomial mapping $\ph : G \times_{G_v} N_v \to V$, $[g,x] \mapsto g(v+x)$, 
where the action $G \acts (G \times_{G_v} N_v)$ is by left multiplication on the first component,  
induces a polynomial mapping $\ps : (G \times_{G_v} N_v)\cq G \to V\cq G$ sending   
$(G \times_{G_v} N_v)/ G$ into $V/ G$. 

The $G_v$-equivariant embedding $\al : N_v \hookrightarrow G \times_{G_v} N_v$ given by $x \mapsto [e,x]$ induces 
an isomorphism $\be : N_v\cq G_v \to (G \times_{G_v} N_v) \cq G$ mapping $N_v/G_v$   
onto $(G \times_{G_v} N_v)/ G$. Set $\et = \ph \o \al$ and $\th = \ps \o \be$.

\[
  \xymatrix{
    N_v \ar@{_(->}[rr]^{\al} \ar[d]_{\ta} \ar@/^20pt/[rrrr]^{\et}
    && G \times_{G_v} N_v \ar[rr]^{\ph} \ar[d] && V \ar[d]^{\si} \\
    N_v/G_v \ar[rr] \ar@{^(->}[d] && (G \times_{G_v} N_v)/G \ar@{^(->}[d] \ar[rr] && V/G \ar@{^(->}[d] \\ 
    N_v \cq G_v \ar[rr]^{\be}  \ar@/_20pt/[rrrr]_{\th}
    && (G \times_{G_v} N_v)\cq G  \ar[rr]^{\ps} && V \cq G 
  } 
\]

\begin{theorem}[{Cf.\ \cite{Luna73}, \cite{Schwarz80}}] \label{thm:slice}
  There is an open ball $B_v \subseteq N_v$ centered at the origin such that the restriction of $\ph$ to 
  $G \times_{G_v} B_v$ is an analytic $G$-isomorphism onto a $G$-invariant neighborhood of $v$ in $V$. 
  The mapping $\th$ is a local analytic isomorphism at $0$ which induces a local homeomorphism of 
  $N_v/G_v$ and $V/G$.
\end{theorem}

\subsection{Reduction} \label{ssec:red}

Let $\{\ta_i\}_{i=1}^m$ be a system of generators of $\R[N_v]^{G_v}$ and  
let $\ta = (\ta_1,\ldots,\ta_m) : N_v \rightarrow \mathbb{R}^m$ be the associated orbit mapping. 
Consider the slice 
\begin{equation}\label{eq:normalslice}
  S_v := v+B_v,  
\end{equation}
where $B_v$ is the open ball from Theorem~\ref{thm:slice}. 
As $\si_i$ is $G_v$-invariant there exists $\pi_i \in \R[\R^m]$ so that
\begin{equation} \label{eq:pi}
  \si_i(x) - \si_i(v) = \pi_i (\ta(x-v)), \quad \text{ for } x \in S_v.  
\end{equation} 
Conversely, every $G_v$-invariant real analytic function in $x-v$ can be written as a real analytic function 
in $\si(x)-\si(v)$ near $v$, by \cite[p.~67]{Schwarz75}, hence 
there is a real analytic mapping $\vh$ defined in a neighborhood of the origin in $\R^n$ with values in $\R^m$
such that 
\begin{equation} \label{eq:vh}
   \ta(x-v) = \vh (\si(x) - \si(v)),  
\end{equation} 
for $x$ in some neighborhood $U_v$ of $v$ in $S_v$.

\begin{lemma} \label{lem:red}
  Let $c=(c_1,\ldots,c_n)$ be a curve in $\si(V)$ with $c_1 \ne 0$ and such that the curve 
  \[
    \underline c := \big(1,{c_1}^{-\frac{d_2}{d_1}} c_2,\ldots,{c_1}^{-\frac{d_n}{d_1}} c_n\big)
  \] 
  lies in $\si(U_v)$.
  Then $\underline c^* := \vh (\underline c -\si(v))$ is a curve in $\ta(U_v-v)$ and 
  \[
    c^* = (c^*_1,\ldots,c^*_m) 
    := ({c_1}^{\frac{e_1}{d_1}} \underline c^*_1,\ldots,{c_1}^{\frac{e_m}{d_1}} \underline c^*_m),
    \quad e_i = \deg \ta_i,
  \]  
  is a curve in $\ta(N_v)$. If $\overline c^*$ is a lift of $c^*$ over $\ta$ then  
  \begin{equation} \label{eq:red}
    {c_1}^\frac{1}{d_1} v + \overline c^*
  \end{equation}
  is a lift of $c$ over $\si$.
\end{lemma}

\begin{proof}
  Only the last statement is maybe not immediately visible. 
  The curve ${c_1}^{-\frac{1}{d_1}} \overline c^*$ is a lift of $\underline c^*$ over $\ta$, 
  \[
    \ta_i({c_1}^{-\frac{1}{d_1}} \overline c^*) = {c_1}^{-\frac{e_i}{d_1}} \ta_i(\overline c^*) 
    = {c_1}^{-\frac{e_i}{d_1}} c^*_i = \underline c^*_i,   
  \]
  and so, by \eqref{eq:pi} and \eqref{eq:vh}, ${c_1}^{-\frac{1}{d_1}} \overline c^* + v$ is a lift of $\underline c$ over $\si$,
  \[
    \si({c_1}^{-\frac{1}{d_1}} \overline c^* + v) -\si(v) = 
    \pi(\ta({c_1}^{-\frac{1}{d_1}} \overline c^*+v-v)) = \pi(\underline c^*) = \pi(\vh (\underline c -\si(v))) = \underline c -\si(v).
  \]
  By homogeneity, we find $\si_i(\overline c^* + {c_1}^{\frac{1}{d_1}} v) = {c_1}^{\frac{d_i}{d_1}} \underline c_i = c_i$ 
  as required.
\end{proof}

We can assume that  
$\vh$ and all its partial derivatives are 
separately bounded.
In analogy to \eqref{eq:red1} we may assume that $\ta_1(x) = \|x\|^2$ for all $x \in N_v$, thus $e_1 =2$.
Then the following corollary is evident. 

\begin{corollary} \label{cor:red}
  We have $|c^*_1| \le C_0 \, |c_1|$, where $C_0 = \sup_{y} |\vh_1(y)|$.
\end{corollary}

The set $\si(V)$ is closed in $\R^n_y$. Thus \eqref{eq:dom} implies that the set $\si(V) \cap \{y_1 = 1\}$
is compact. It follows that the open cover $\{\si(U_v)\}_{v \in V, \|v\|=1}$ of $\si(V) \cap \{y_1 = 1\}$ has a 
finite subcover 
\begin{equation} \label{eq:red5}
  \{B_\al\}_{\al \in \De} = \{\si(U_{v_\al})\}_{\al \in \De}.  
\end{equation}

The following lemma shows that the maximal degree of the basic invariants does not increase by 
passing to a slice representation. This was shown in \cite[Lemma 2.4]{KLMR06}; for convenience of the reader we 
include a short proof. 

\begin{lemma} \label{lem:e}
  Assume that $\{\ta_i\}_{i=1}^m$ is minimal and set $e := \max_{i} e_i = \max_{i} \deg \ta_i$. Then $e \le d$.
\end{lemma}

\begin{proof}
  We may assume without loss of generality that the basic invariants $\ta_i$ are ordered so that 
  $e_1\le e_2\le \cdots \le e_m = e$. 
  Assume that $e_m > d$. We will show that this assumption contradicts minimality of 
  $\{\ta_i\}_{i=1}^m$. 
  It fact, in view of \eqref{eq:pi} it implies that each polynomial $\pi_i$ is independent of its last entry. 
  Thus, by \eqref{eq:pi} and \eqref{eq:vh}, we have for $y \in U_v-v$,
  \[
    \ta_m(y) = \ps_m(\ta'(y)),  
  \]
  where $\ta':= (\ta_1,\ldots,\ta_{m-1})$ and $\ps_m := \vh_m \o \pi$. Expanding into Taylor series at $0$,  
  \[
    \ta_m = T^\infty_0 \ps_m \o \ta' = T^{e}_0 \ps_m \o \ta', 
  \]
  we see that $\ta_m$ is a polynomial in $\ta_1,\ldots,\ta_{m-1}$ (in a neighborhood of $0$ and hence everywhere in $N_v$). 
  This contradicts minimality of $\{\ta_i\}_{i=1}^m$.
\end{proof}

\section{Two interpolation inequalities}

We recall two classical interpolation inequalities.
The first is a version of Glaeser's inequality (cf.\ \cite{Glaeser63R}).

\begin{lemma} \label{Glaeser}
  Let $I\subseteq \R$ be an open interval and   
  let $f \in C^{1,1}(\overline I)$ be nonnegative. 
  For any $t_0 \in I$ and $M>0$ such that $I_{t_0}(M^{-1}) := \{t : |t-t_0|< M^{-1}|f(t_0)|^{\frac 1 2}\} \subseteq I$
  and $M^2 \ge \Lip_{I_{t_0}(M^{-1})}(f')$ 
  we have 
  \[
    |f'(t_0)| \le \big(M + M^{-1} \on{Lip}_{I_{t_0}(M^{-1})}(f')\big) |f(t_0)|^{\frac{1}{2}} \le 2M |f(t_0)|^{\frac{1}{2}}. 
  \]
\end{lemma}

\begin{proof}
  The inequality holds true at zeros of $f$. Let us assume that $f(t_0)>0$. 
  The statement follows from  
  \[
   0 \le f(t_0 +h) = f(t_0) + f'(t_0) h + \int_0^1 (1-s) f''(t_0+ h s)\, ds\, h^2
  \]
  with  
  $h = \pm M^{-1}|f(t_0)|^{\frac{1}{2}}$.
\end{proof}

\begin{lemma} \label{taylor} 
Let $f\in C^{m-1,1}(\bar I)$.  There is a universal constant $C=C(m)$ such that for all $t\in I$ and $k = 1,\ldots,m$,  
    \begin{align}\label{eq:1}  
    |f^{(k)}(t) | \le C |I|^{-k} \bigl(\| f \|_{L^\infty(I)}+  \on{Lip}_{I}(f^{(m-1)})  |I|^m \bigr).  
  \end{align}
\end{lemma}

\begin{proof}
We may suppose $I=(-\delta, \delta)$. If $t \in I$ then at least one of the two intervals
 $[t,t\pm \delta )$, say $[t,t+ \delta )$, is included in $I$.  By Taylor's formula, for $t_1\in [t,t +\delta )$, 
    \begin{align*}
      \Big|\sum_{k=0}^{m-1}  \frac{{f}^{(k)}(t)}{k!} (t_1-t)^k\Big| 
       & \le |f(t_1)| +  \int_0^1 \frac{(1-s)^{m-1}}{(m-1)!} |f^{(m)}(t + s(t_1-t))|\, ds\, (t_1-t)^m  \\
      & \le \| f \|_{L^\infty(I)}+  \on{Lip}_{I}(f^{(m-1)})  \delta^m ,        
    \end{align*}
   and for $k \le m-1$  we may conclude by Proposition~\ref{prop:interpol} below.   
    For $k=m$,  \eqref{eq:1} is trivially satisfied.  
\end{proof}

\begin{proposition} \label{prop:interpol}
  Let $P(x) = a_0 + a_1 x + \cdots + a_m x^m \in \C[x]$ satisfy 
  $|P(x)| \le A$ for $x \in [0,B] \subseteq \R$. Then, for $j=0,\ldots,m$,
  \[
    |a_j| \le (2m)^{m+1} A B^{-j}.
  \]
\end{proposition}

\begin{proof}
  We show the lemma for $A=B=1$.
  The general statement follows by applying this special case to the polynomial $A^{-1} P(By)$, $y=B^{-1}x$.
  Let $0 = x_0 < x_1 < \cdots < x_m=1$ be equidistant points. By Lagrange's interpolation formula 
  (e.g.\ \cite[(1.2.5)]{RS02}), 
  \[
    P(x) = \sum_{k=0}^m P(x_k) \prod_{\substack{j=0\\ j\ne k}}^m \frac{x-x_j}{x_k-x_j},
  \]
  and therefore
  \[
    a_j = \sum_{k=0}^m P(x_k) \prod_{\substack{j=0\\ j\ne k}}^m (x_k-x_j)^{-1} (-1)^{m-j} \si^k_{m-j},
  \]
  where $\si^k_j$ is the $j$th elementary symmetric polynomial in $(x_\ell)_{\ell \ne k}$.
  The statement follows. 
\end{proof}

A better constant can be obtained using Chebyshev polynomials; cf.\ \cite[Theorems~16.3.1-2]{RS02}.

\section{Proof of Theorem~\ref{thm:main}} \label{sec:proof}
 
Let $(G\acts V,d,\si)$ satisfy \eqref{eq:red1} and \eqref{eq:red2}, and  
let $c \in C^{d-1,1}(I,\si(V))$.

\subsection{Reduction to \texorpdfstring{$G \acts (V\setminus \{0\})$}{G:V-0}}

By \eqref{eq:red1} we have $c_1\ge 0$ and $c_1(t) = 0$ if and only if $c(t) = 0$.
We shall show the following statement.

\begin{claim} \label{claim:a}
  For any relatively compact open subinterval $I_0 \Subset I$ and any 
  $t_0 \in I_0 \setminus {c_1}^{-1}(0)$, 
  there exists a Lipschitz lift $\overline c_{t_0}$ of $c$ on a neighborhood $I_{t_0}$ of $t_0$ in 
  $I_0 \setminus {c_1}^{-1}(0)$ so that 
  \[
    \on{Lip}_{I_{t_0}}(\overline c_{t_0}) \le C\, \big(\max_i \|c_i\|^{\frac{1}{d_i}}_{C^{d-1,1}(\overline I_1)}\big), 
  \]
  where $I_1$ is any open interval 
  satisfying $I_0 \Subset I_1 \Subset I$ and $C=C(G\acts V,I_0,I_1)$.
\end{claim}

Claim~\ref{claim:a} will imply Theorem~\ref{thm:main} by the following lemma. 

\begin{lemma} \label{lem:ext} 
  Suppose that for each $t_0 \in I_0 \setminus {c_1}^{-1}(0)$ there exists a Lipschitz lift $\overline c_{t_0}$ of $c$
  on a neighborhood $I_{t_0}$ of $t_0$ in $I_0 \setminus {c_1}^{-1}(0)$ so that 
  $L := \sup_{t_0 \in I_0 \setminus {c_1}^{-1}(0)} \on{Lip}_{I_{t_0}}(\overline c_{t_0}) < \infty$. 
  Then there exists a Lipschitz lift $\overline c$ of $c$
  on $I_0$ and $\on{Lip}_{I_0}(\overline c) \le L$.
\end{lemma}

\begin{proof}
  Let $J$ be any connected component of $I_0 \setminus {c_1}^{-1}(0)$. 
  If $\overline c_i$, $i=1,2$, are local Lipschitz lifts of $c$ defined on subintervals $(a_i,b_i)$, $i=1,2$, of $J$ with 
  $a_1 < a_2 < b_1 < b_2$ and so that $\on{Lip}_{(a_i,b_i)}(\overline c_i) \le L$, $i=1,2$, then there exists a Lipschitz 
  lift $\overline c_{12}$ of $c$ on $(a_1,b_2)$ satisfying $\on{Lip}_{(a_1,b_2)}(\overline c_{12}) \le L$. 
  To see this choose a point $t_{12} \in (a_2,b_1)$. Since $G \overline c_1(t_{12}) = G \overline c_2(t_{12})$, 
  there exists $g_{12} \in G$ so that $\overline c_1(t_{12}) = g_{12} \overline c_2(t_{12})$. 
  Define $\overline c_{12}(t) := \overline c_1(t)$ for $t\le t_{12}$ and $\overline c_{12}(t) := g_{12} \overline c_2(t)$ for $t\ge t_{12}$. 
  It is easy to see that $c_{12}$ has the required properties (since $G$ acts orthogonally). 

  These arguments imply that there exists a Lipschitz lift $\overline c_J$ of $c$ with  
  $\on{Lip}_{J}(\overline c_J) \le L$ 
  on each connected component $J$ of $I_0 \setminus {c_1}^{-1}(0)$. Defining $\overline c(t):=\overline c_J(t)$ if $t \in J$ and  
  $\overline c(t):=0$ if $t \in {c_1}^{-1}(0)$, we obtain a continuous lift of $c$, since $c_1(t) = \|\overline c(t)\|^2$, by \eqref{eq:red1}. 
  It is easy to see that $\on{Lip}_{I_0}(\overline c) \le L$. 
\end{proof}

Let us prove that Claim~\ref{claim:a} and Lemma~\ref{lem:ext} imply Theorem~\ref{thm:main}. 
That they imply Remark~\ref{rem:main} is obvious. 
Let $J_1 \subseteq J_2 \subseteq \cdots $ be a countable exhaustion of $I$ by compact intervals so that, for all $k$, $J_k$ is contained 
in the interior of $J_{k+1}$. By Claim~\ref{claim:a} and Lemma~\ref{lem:ext}, there exist lifts 
$\overline c_k : J_k \to V$, $k\ge 1$, of $c$ and compact neighborhoods $K_k \supseteq J_k$ in $I$ so that 
\[
    \on{Lip}_{J_k}(\overline c_k) \le C\, \big(\max_i \|c_i\|^{\frac{1}{d_i}}_{C^{d-1,1}(K_k)}\big), \quad k \ge 1,
\]
for $C=C(G \acts V, J_k, K_k)$.
We may construct a $C^{0,1}$-lift $\overline c : I \to V$ of $c$ iteratively in the following way. 
If $\overline c$ already exists on $J_k$ we extend it on $J_{k+1} \setminus J_k$ by $g \overline c_{k+1}$ for suitable 
$g \in G$ left and right of $J_k$ (cf.\ the first paragraph of the proof of Lemma~\ref{lem:ext}).
If $I_0  \Subset I$ is relatively compact then $I_0 \subseteq  J_{N}$ for some $N$. 
Thus for $t,s \in I_0$, $t<s$, there is a sequence $t=:t_0 < t_1 <\cdots <t_\ell := s$ of endpoints $t_i$ of the intervals 
$J_k$ (except possibly $t_0$ and $t_\ell$), elements $g_i \in G$, and $k_i \in \{1,\ldots,N\}$ 
so that 
\[
    \|\overline c(t)-\overline c(s))\| \le \sum_{i=1}^\ell \|g_i\overline c_{k_i}(t_i)-g_i\overline c_{k_i}(t_{i-1})\|
    = \sum_{i=1}^\ell \|\overline c_{k_i}(t_i)-\overline c_{k_i}(t_{i-1})\|
     \le \max_{1 \le k\le N} \on{Lip}_{J_k}(\overline c_k)\, |t-s|. 
\]
Setting $I_1:= \cup_{k=1}^N K_k$ we obtain \eqref{eq:uniform}.

\subsection{Convenient assumption}

The proof of Claim~\ref{claim:a} will be carried out by induction on the \emph{size} of $G$. If $G$ and $H$ 
are compact Lie groups we write $H < G$ if and only if $\dim H < \dim G$ or, if $\dim H = \dim G$, $H$ has fewer 
connected components than $G$.    

We replace the assumption that $c \in C^{d-1,1}(I,\si(V))$ by a new (weaker) assumption 
that will be more convenient for the inductive step. Before stating it we need a bit of notation. 

For open intervals $I_0$ and $I_1$ so that $I_0 \Subset I_1 \Subset I$, we set 
\[
  I_i':= I_i \setminus {c_1}^{-1}(0), \quad i=0,1.
\]
For $t_0 \in I_0'$ and $r>0$ consider the interval 
\[
  I_{t_0}(r) := \big(t_0 -r |c_1(t_0)|^\frac{1}{2},t_0 + r |c_1(t_0)|^\frac{1}{2}\big).
\]

\begin{assumption*} 
  Let $I_0 \Subset I_1$ be open intervals.      
  Suppose that $c \in C^{d-1,1}(\overline I_1,\si(V))$ and assume that there is a constant $A>0$ 
  so that for all $t_0 \in I_0'$, $t \in I_{t_0}(A^{-1})$, $i = 1,\ldots,n$, $k = 0,\ldots,d$,
  \begin{gather}
    \tag{A.1}\label{A.2} I_{t_0}(A^{-1}) \subseteq I_1\\ 
    \tag{A.2}\label{A.3} 2^{-1} \le \frac{c_1(t)}{c_1(t_0)} \le 2 \\ 
    \tag{A.3}\label{A.4} |{c_i}^{(k)}(t)| \le C\, A^{k}\, |c_1(t)|^{\frac{d_i-k}{d_1}} 
  \end{gather}
  where $C = C(G \acts V) \ge 1$.
  For $k=d$, \eqref{A.4} is understood to hold almost everywhere, by Rademacher's theorem. 
\end{assumption*}

\begin{remark*}    
  Condition \eqref{A.4} implies that  
  \begin{equation}
    \tag{A.4}\label{A.6} \big|\p_t^k \big({c_1}^{- \frac {d_i}{d_1}} c_i\big)(t)\big| \le C\, A^k\, |c_1(t)|^{-\frac k{d_1}},
  \end{equation}  
  where $C=C(G\acts V)$.
  In fact, if we assign $c_i$ the weight $d_i$ (and ${c_1}^{\frac 1 {d_1}}$ the weight $1$) 
  and let 
  $L(x_1,\ldots,x_n,y) \in \R[x_1,\ldots,x_n,y,y^{-1}]$ be weighted homogeneous of degree $D$, 
  then    
  \[
    \big|\p_t^k L\big(c_1,\ldots,c_n,{c_1}^{\frac 1 {d_1}}\big)(t)\big| 
    \le C\, A^k\, |c_1(t)|^{\frac{D-k}{d_1}},
  \] 
  for $C =C(G\acts V,L)$.
\end{remark*}

The following two claims clearly imply Claim~\ref{claim:a}.

\begin{claim} \label{claim:b}
  Any curve $c \in C^{d-1,1}(\overline I_1,\si(V))$ satisfying  
  \eqref{A.2}--\eqref{A.4} has a Lipschitz lift on a neighborhood of any $t_0 \in I_0'$ 
  with Lipschitz constant bounded from above by $C\, A$, where $C=C(G \acts V)$.
\end{claim}

\begin{claim} \label{claim:c}
  If $c \in C^{d-1,1}(I,\si(V))$ then \eqref{A.2}--\eqref{A.4} hold 
  for each pair of open intervals $I_0$ and $I_1$ satisfying $I_0 \Subset I_1 \Subset I$ and with 
  $A \le C\, (\max_i \|c_i\|^{\frac{1}{d_i}}_{C^{d-1,1}(\overline I_1)})$ for $C=C(I_0, I_1)$.
\end{claim}

\subsection{Proof of Claim~\ref{claim:b} (inductive step)}\label{proofofiii}

Let $c$, $I_0$, $I_1$, $A$, $t_0$ be as in the Assumption and hence satisfy \eqref{A.2}--\eqref{A.4}. 
We will show the following.
\begin{itemize}
  \item For some constant $C_1=C_1(G\acts V)>1$, the lifting problem for $c$ reduces on the interval 
  $I_{t_0}({C_1}^{-1}A^{-1})$ 
  to the lifting problem for some 
  associated curve $c^*$ in the orbit space of some slice representation $H\acts W$ of $G\acts V$ with  $H<G$.  
  \item The curve $c^*$ satisfies \eqref{A.2}--\eqref{A.4} for suitable neighborhoods $J_0$, $J_1$ of $t_0$
  and a constant $B=C\, A$ in place of $A$, where $C=C(G\acts V)$ . 
\end{itemize} 
This will allow us to conclude Claim~\ref{claim:b} by induction on the size of $G$.

Let us restrict $c$ to $I_{t_0}(A^{-1})$ and consider  
\[
  \underline c:= \big(1,{c_1}^{-\frac{d_2}{d_1}} c_2,\ldots,{c_1}^{-\frac{d_n}{d_1}} c_n\big) : 
  I_{t_0}(A^{-1}) \to  \si(V) \subseteq \R^n_y.
\]
Then $\underline c$ is continuous, by \eqref{A.3}, and bounded, by \eqref{eq:dom}. 
Moreover, by \eqref{A.6} and \eqref{A.3}, for $t \in I_{t_0}(A^{-1})$,
\begin{equation} \label{eq:barader}
  \|\underline c'(t)\| \le C_1\, A\, |c_1(t_0)|^{-\frac 1 {d_1}},    
\end{equation}  
for $C_1=C_1(G \acts V)$.
Consider the finite open cover $\{B_\al\}_{\al \in \De}= \{\si(U_{v_\al})\}_{\al \in \De}$ of the compact set $\si(V) \cap \{y_1=1\}$ 
from \eqref{eq:red5}. Let $2r_1>0$ be a Lebesgue number of the cover $\{B_\al\}_{\al \in \De}$.
Then for any $p \in \si(V) \cap \{y_1 = 1\}$ there is $\al_p \in \De$ so that 
\[
  B_p(r_1) \cap \si(V) \cap \{y_1=1\}  \subseteq B_{\al_p},
\]
where $B_p(r_1) \subseteq \R^n$ is the open ball centered at $p$ with radius $r_1$. 
If $C_1$ is the constant from \eqref{eq:barader}, then
\begin{equation} \label{eq:J1}
  J_1 := I_{t_0}(r_1 {C_1}^{-1} A^{-1})  \subseteq {\underline c}^{-1} (B_{\underline c(t_0)}(r_1)).
\end{equation}
By Lemma~\ref{lem:red} the lifting problem on the interval $J_1$ reduces to the curve 
$c^* = (c^*_i)_{i=1}^m$,
\begin{equation} \label{eq:cstar}
   c^*_i = c_1^{\frac{e_i}{d_1}} \vh_i\big({c_1}^{-\frac{d_2}{d_1}} c_2,\ldots,{c_1}^{-\frac{d_n}{d_1}} c_n\big),
   \quad e_i = \deg \ta_i,   
\end{equation}  
in $\ta(N_v)$, where $G_v \acts N_v$ is the slice representation at $v= v_{\al_{\underline c(t_0)}}$ with orbit 
mapping $\ta=(\ta_1,\ldots,\ta_m)$ 
and where the $\vh_i$ are real analytic; 
the first summand of \eqref{eq:red} 
is Lipschitz with Lipschitz constant 
bounded from above by $C\, A$ with $C=C(G \acts V)$ thanks to \eqref{A.4}. 
Fix $r_0 < r_1$ and set 
\begin{equation} \label{eq:J0}
  J_0 := I_{t_0}(r_0 {C_1}^{-1} A^{-1}),
\end{equation}
where $C_1$ is the constant from \eqref{eq:barader}.
(Here we assume without loss of generality that $r_1< C_1$ so that $r_0 {C_1}^{-1} < r_1 {C_1}^{-1} < 1$ and 
hence $J_0 \subseteq J_1 \subseteq I_{t_0}(A^{-1})$.)

Let us show that the curve $c^*$ satisfies \eqref{A.2}--\eqref{A.4} 
for the intervals $J_1$ and $J_0$ from \eqref{eq:J1} and \eqref{eq:J0} and a suitable constant $B>0$ in place of $A$. 
To this end we set 
\[
  J_i':= J_i \setminus (c^*_1)^{-1}(0), \quad i=0,1,
\]
consider, 
for $t_1 \in J_0'$ and $r>0$, the interval
\[
  J_{t_1}(r) := \big(t_1 -r |c^*_1(t_1)|^\frac{1}{2},t_1 + r |c^*_1(t_1)|^\frac{1}{2}\big),
\]
and prove the following lemma.

\begin{lemma}
  There is a constant $C=C(G\acts V,r_1,r_0)>1$ such that for $B = C\, A$ and 
  for all $t_1 \in J_0'$, $t \in J_{t_1}(B^{-1})$, $i = 1,\ldots,m$, $k = 0,\ldots,d$,
  \begin{gather}
    \tag{B.1}\phantomsection\label{B.2} J_{t_1}(B^{-1}) \subseteq J_1\\ 
    \tag{B.2}\phantomsection\label{B.3} 2^{-1} \le \frac{c^*_1(t)}{c^*_1(t_1)} \le 2 \\ 
    \tag{B.3}\phantomsection\label{B.4} |{(c^*_i)}^{(k)}(t)| \le \tilde C\, B^k\, |c^*_1(t)|^{\frac{e_i-k}{e_1}} 
  \end{gather}
  where $\tilde C=\tilde C(G\acts V)$.
\end{lemma}

\begin{proof}
  If
  \begin{equation*} \label{eq:D1}
    B \ge  (r_1-r_0)^{-1} \sqrt{2\, C_0}\, C_1\,  A,
  \end{equation*}
  where $C_0$ and $C_1$ are the constants from Corollary~\ref{cor:red} and \eqref{eq:barader}, respectively,
  then  
  by Corollary~\ref{cor:red} and \eqref{A.3},
  \[
    B^{-1}|c^*_1(t_1)|^{\frac 1 2}  
    \le (r_1-r_0)\, {C_1}^{-1}\,  A^{-1}\, |c_1(t_0)|^{\frac 1 2},
  \]
  and so \eqref{B.2} follows from \eqref{eq:J1} and \eqref{eq:J0}, as $t_1 \in J_0$. 

  Next we claim that, on $J_1$,
  \begin{equation} \label{eq:psi}
    \big|\p_t^k \vh_i \big({c_1}^{-\frac{d_2}{d_1}} c_2,\ldots,{c_1}^{-\frac{d_n}{d_1}} c_n\big)\big| 
    \le C\, A^k\, |c_1|^{-\frac k {d_1}},
  \end{equation}
  for $C=C(G\acts V)$.
  To see this we differentiate the following equation $(k-1)$ times, apply  
  induction on $k$, and use \eqref{A.6},
  \begin{equation} \label{derivatives}
    \p_t \vh_i \big({c_1}^{-\frac{d_2}{d_1}} c_2,\ldots,{c_1}^{-\frac{d_n}{d_1}} c_n\big) 
    = \sum_{j=1}^n (\p_j \vh_i)(\underline c)\, \p_t \big({c_1}^{-\frac {d_j}{d_1}} c_j\big);
  \end{equation}
  recall that all partial derivatives of the $\vh_i$'s are separately bounded on $\underline c(J_1)$ and these bounds are 
  universal. From \eqref{eq:cstar} and \eqref{eq:psi} we obtain, on $J_1$ and for all $i=1,\ldots,m$, $k=0,\ldots,d$,
  \begin{equation} \label{derivatives2}
    |{(c^*_i)}^{(k)}| \le C\, A^k\, |c_1|^{\frac{e_i-k}{d_1}},
  \end{equation}
  for $C=C(G\acts V)$,
  and so, by Corollary~\ref{cor:red} and as $d_1=e_1=2$,   
  \begin{equation} \label{eq:i-k<0}
    |{(c^*_i)}^{(k)}| \le C\, A^k\, |c^*_1|^{\frac{e_i-k}{e_1}} \quad \text{ if } e_i-k \le 0,
  \end{equation} 
  for $C = C(G\acts V)$. 
  This shows \eqref{B.4} for $k\ge e_i$, and \eqref{B.4} for $k=0$ follows from \eqref{eq:dom}.
  The remaining inequalities, i.e., \eqref{B.4} for $0<k< e_i$ as well as \eqref{B.3}, follow now from Lemma \ref{shortcut} below
  (since $d \ge e = \max_i e_i$, by Lemma~\ref{lem:e}).
\end{proof}

\begin{lemma}\label{shortcut}
There is a constant $C=C(G \acts V)\ge 1$ such that the following holds.  
If \eqref{A.2} and \eqref{A.4} for $k=0$ and $k=d_i$, $i=1, ... ,n$,
are satisfied, then so are \eqref{A.3} and \eqref{A.4} for $k<d_i$, $i=1, ... ,n$, 
after replacing $A$ by $C\, A$.   
\end{lemma}

\begin{proof}
By assumption $\on{Lip}_{I_{t_0}(A^{-1})}(c_1')\le C\, A^2$, where $C$ is the constant from \eqref{A.4}.  
Thus, by Lemma \ref{Glaeser} for  $f=c_1$ and  $M = C^{\frac 1 2} A$, we 
get 
$$
|c_1'(t_0)| \le 2 M |c_1 (t_0)|^{\frac 1 2}.  
$$
It follows that, for $t\in I_{t_0}((6M)^{-1})$, 
\begin{align}\label{long}
  \frac{|c_1(t) - c_1(t_0)|}{|c_1(t_0)|} \le 
      \frac{|c_1'(t_0)|}{|c_1(t_0) |}|t-t_0| + \int_0^1 (1-s) |c_1''(t_0 + s (t-t_0))| ds\, \frac{|t-t_0|^2}{|c_1(t_0)|} 
      \le \frac{1}{2}
    \end{align}
which implies \eqref{A.3}.  
The other inequalities follow from Lemma \ref{taylor}.  
\end{proof}

We may now finish the proof of Claim \ref{claim:b}.
By assumption \eqref{eq:red2}, $V^G =\{0\}$ and thus $G_v < G$. The inductive hypothesis yields a Lipschitz lift 
$\overline c^*$ of $c^*$ over $\ta$ with Lipschitz constant bounded from above by $C\, B$, 
for $C=C(G_v\acts N_v)$. 
By Lemma~\ref{fix} and \eqref{eq:i-k<0} for $e_i=k=1$ (the basic invariants of $G_v \acts N_v^{G_v}$ form a system of linear coordinates 
on $N_v^{G_v}$), 
we can assume that $N_v^{G_v} =\{0\}$.
By Lemma~\ref{lem:red},
\[
    {c_1}^\frac{1}{d_1} v + \overline c^*
\]
is a lift of $c$ over $\si$. 
Thanks to \eqref{A.4} for $i=k=1$ and 
since there are only finitely many isomorphism types of slice representations, this lift is Lipschitz with Lipschitz 
constant bounded 
from above by $C\, A$, for $C=C(G \acts V)$.
This ends the proof of Claim \ref{claim:b}.

\subsection{Proof of Claim \ref{claim:c}} \label{proof:iii} 

Let $\de$ denote the distance between the endpoints of $I_0$ and those of $I_1$. 
  Set 
  \begin{align}\label{A1A2}
     & A_1:=  \max\Big\{  
      \de^{-1}\|c_1\|_{L^\infty(I_1)}^{\frac{1}{2}}, (\on{Lip}_{I_1}(c_1'))^{\frac 1 2}  \Big\}  \\ \notag
     & A_2 :=  \max_i \Big\{M_i \|c_1\|^{\frac {d-d_i} 2}_{L^\infty(I_1)} \Big\} ^{\frac 1 d} ,  \quad
     M_i :=   \on{Lip}_{I_1}({c_i}^{(d-1)}), 
  \end{align}
  and choose 
  \begin{align}\label{formulaforA}
    A \ge  A_0 = 6 \max \{{A_1}, A_2\}. 
  \end{align} 
  To have \eqref{A.2} and \eqref{A.3} it suffices to assume $A\ge 6A_1$. 
  For $t_0 \in I_0'$ obviously $I_{t_0}({A_1}^{-1}) \subseteq I_1$ and thus \eqref{A.2}. 
  Then Lemma~\ref{Glaeser} implies  
    \begin{align*}
       |c_1'(t_0)| \le  2A_1 \, |c_1(t_0)|^{\frac{1}{2}},
    \end{align*} 
  and so, for $t_0 \in I_0'$ and $t \in I_{t_0}((6A_1)^{-1})$,  
  \eqref{long} and hence \eqref{A.3} holds. 
  Finally, 
  Lemma~\ref{taylor}, \eqref{eq:dom}, and \eqref{A.3} imply \eqref{A.4} for $t \in I_{t_0}(A^{-1})$.

\subsection{Bounds for the Lipschitz constant} \label{ssec:bounds}

Let $\rep$ satisfy \eqref{eq:red1} and \eqref{eq:red2}, let $c \in C^{d-1,1}(I,\si(V))$, and let $I_0 \Subset I$.
Then there is a neighborhood $I_1$ of $I_0$ with $I_0 \Subset I_1 \Subset I$ such that the lift 
$\overline c \in C^{0,1}(I,V)$ constructed in the above proof satisfies
\begin{align}\label{finalbounds}
  \on{Lip}_{I_0}(\overline c)  & \le  C(G \acts V)\, 
    \max\Big\{  \de^{-1}\|c_1\|_{L^\infty( I_1)}^{\frac{1}{2}}, (\on{Lip}_{ I_1}(c_1'))^{\frac 1 2},   
    \max_i \big\{M_i \|c_1\|^{\frac {d-d_i} 2}_{L^\infty( I_1)} \big\} ^{\frac 1 d}   \Big\} \\
\notag  
     &  \le C(G \acts V,I_0,I_1)\, \big(\max_i \|c_i\|^{\frac 1 {d_i}}_{C^{d-1,1}(\overline I_1)}\big) 
      \\
\notag  
     &  \le C(G \acts V,I_0,I_1)\, \big(1+\max_i \|c_i\|_{C^{d-1,1}(\overline I_1)}\big) 
\end{align}
where $\de$ is the distance between the endpoints of $I_0$ and those of $I_1$, and 
$M_i = \on{Lip}_{ I_1}({c_i}^{(d-1)})$. 
This follows from Claim \ref{claim:b}, \eqref{A1A2}, \eqref{formulaforA}, and Lemma \ref{lem:ext}.

\section{Proof of Theorem~\ref{mainC1}} \label{sec:proofC1}

Let $(G\acts V,d,\si)$ satisfy \eqref{eq:red1} and \eqref{eq:red2}, and  
let $c \in C^{d}(I,\si(V))$.
In the proof of Theorem~\ref{mainC1}, induction on the size of $G$ will provide us with local lifts of class $C^1$ 
near points where $c$ is not flat (in the sense that they are not of Case~\hyperref[Case2]{2} of Subsection \ref{endofproof}).  
Moreover, we shall see that the derivatives of these local lifts converge to $0$ as $t$ tends 
to flat points. This faces us with the problem of gluing these local lifts. We tackle this problem first.

\subsection{Algorithm for local lifts} \label{ssec:algorithm}

We choose a finite cover 
$\{G U_{v_\al}\}_{\al \in \De}$  
of a neighborhood of the sphere $S(V) ={c_1}^{-1}(1)$ in $V$ so that $U_v$ is transverse to all the orbits in 
$GU_{v_\al}$ with the angle very close to $\pi / 2$.  It induces a cover of $\si(V) \cap \{y_1=1\}$,  
 \begin{equation*}%
  \{B_\al\}_{\al \in \De} = \{\si(U_{v_\al})\}_{\al \in \De},  
\end{equation*}
in analogy to \eqref{eq:red5}.

Lemma \ref{lem:red} provides an algorithm for the construction of a lift of $c$.   
After removing the fixed points, see Subsection \ref{ssec:fixedpoints},  
we lift $c$ restricted to $I':=\{ t \in I : c_1(t) \ne 0\}$ and then extend it trivially to $\{ t \in I : c_1(t)=0\}$.  
For this we consider 
\[
    \underline c := \big(1,{c_1}^{-\frac{d_2}{d_1}} c_2,\ldots,{c_1}^{-\frac{d_n}{d_1}} c_n\big). 
\] 
For each connected component $I_1$ of the induced cover $  \{{\underline c} ^{-1} (B_\al)\}_{\al \in \De}$ of $I'$  
we lift $c|_{{I_1}}$ to the slice $N_v$, $v=v_\alpha$,  using Lemma \ref{lem:red} and hence the induction.  
This reduction ends when $\underline c(I) \subseteq B_\alpha$ with $B_\alpha$ in the open stratum 
(where we keep the notation $\underline c$, $I$, and $B_\al$ 
for the respective reduced objects). 

Thus for any $t_0\in I$ there is a neighborhood $I_{t_0}$ and a lift $\bar c$ of $c$ on $I_{t_0}$ 
that is entirely contained in an affine transverse slice to the orbit over $c(t_0)$ that is close to the normal 
slice $S_{\overline c(t_0)}$ from \eqref{eq:normalslice}. (Note that the orbit over $0 \in \si(V)$ is just the 
origin in $V$ and every slice is a neighborhood of the origin.)  

This picture is not complete. One needs to make precise how these local lifts are glued together.

\subsection{Change of slice diffeomorphisms}  \label{ssec:change}

Fix $v \in V$ and let $S_v$ be the normal slice of the orbit $Gv$ at $v$; see \eqref{eq:normalslice}.  
Let $H=G_v$ and fix a local analytic section $\varphi_H : G/H \to G$ of the principal bundle $G \to G/H$ 
such that $\varphi_H ([e]) = e$.  Then 
\begin{align}\label{Phi_v}
\Phi_v :   G/H \times S_v  \to V, \quad \Phi_v([g], x) = \varphi_H ([g])  x
\end{align}
is a local diffeomorphism and $\Phi_v  ([e],v) = v$.  Indeed, $\Phi_v$ equals the following composition 
\[
  \xymatrix{
    G/H \times S_v  \ar[rr]^{\al} 
    && G \times_{G_v} S_v \ar[rr]^{\ph} && V ,
  } 
\]
where $\ph : G \times_{G_v} S_v \to V$, $[g,x] \mapsto gx$, is the slice mapping from Theorem \ref{thm:slice}, and 
$\alpha ([g],x)$ is the class of $(\varphi_H([g]),x)$.  Then $\alpha$ is a diffeomorphism with the inverse 
$$
\alpha^{-1} ([g,x]) = \alpha^{-1} \big([g  g^{-1}  \varphi_H([g]),  (\varphi_H([g]))^{-1} g x]\big) 
= ([g], (\varphi_H([g]))^{-1} g x).  
$$

Let $M_v$ 
be another affine transverse slice at $v$, and we suppose that the angle between $N_v$ and $M_v$  is small.  
The second coordinate of the inverse of $\Phi_v$ restricted to $M_v$ gives a local diffeomorphism 
$$
h_{M_v} : M_v \to S_v . 
$$  
The first coordinate of the inverse of $\Phi_v$ composed with $\varphi_H$ gives a mapping 
$$
s_{M_v} : M_v \to G 
$$
such that 
$$
h_{M_v} (x) = (s_{M_v}(x))^{-1}  x.
$$
By \eqref{Phi_v} the partial derivatives of $s_{M_v}$ and $h_{M_v}$ can be bounded in terms of the partial derivatives 
of $\varphi_H $ and the angle between $N_v$ and $M_v$.  

\begin{remark} 
The above construction is uniform in the following sense.  If $v'= g_0 v$ then $H=G_v$ and $H'=G_{v'}$ are conjugate, $H'= g_0 H g_0^{-1}$.  
Conjugation by $g_0$ on $G$ induces an isomorphism $G/H \to G/H', [g]_H \mapsto [g_0 g g_0^{-1}]_{H'}$. 
Given $\varphi_H$ we define $\varphi_{H'}$ by the following diagram.
\[
  \xymatrix{
     G \ar^{\on{conj}_{g_0}}[rr] \ar[d] && G \ar[d] \\
     G/H \ar@/^10pt/^{\varphi_H}[u] \ar^{\cong}[rr]  && G/H' \ar@/_10pt/_{\varphi_{H'}}@{.>}[u]
  }
\]  
Thus 
if we fix $\varphi_H $ for each conjugacy class and suppose  the angle between $N_v$ and $M_v$ is small we 
obtain bounds on the derivatives of  $s_{M_v}$ and $h_{M_v}$ independent of $v$ (valid in a neighborhood of 
$v$ whose size depends on the orbit $Gv$).  
\end{remark}

\subsection{Gluing the local lifts} \label{ssec:glue}

Suppose that there are local lifts $\overline c_1$ and $\overline c_2$ of $c$ resulting from the algorithm described in 
Subsection \ref{ssec:algorithm} such that the respective domains of definition $I_1$ and $I_2$ have nontrivial 
intersection. 
Fix $t_0 \in I_1 \cap I_2$.  We may assume that 
$\overline c_1 (t_0) = \overline c_2 (t_0)$ and denote this vector by $v$.  
Then, by construction, there exist a neighborhood 
$I_{t_0}$ of $t_0$ in $I_1 \cap I_2$ and 
slices $M^1_v$ and $M^2_v$ transverse to $Gv$ containing $\overline c_1(I_{t_0})$ and $\overline c_2(I_{t_0})$, respectively.  
Then, by Subsection \ref{ssec:change}, 
$$
I_{t_0} \ni t \mapsto h_{M^i_v} (\overline c_i (t)), \quad  i=1,2,
$$
are two lifts of $c$ on $I_{t_0}$ contained in $S_v$.  
If we moreover assume that  $c(I_{t_0})$ belongs to a single stratum, then these two lifts coincide 
(since all orbits of type $(G_v)$ meet $S_v$ in a single point), 
and thus, for $t \in I_{t_0}$,
\begin{align}\label{gluing}
s_{M^1_v}  (\overline c_1 (t))^{-1} \ \overline c_1(t)  =  s_{M^2_v}  ( \overline c_2(t))^{-1}\  \overline c_2(t) .
\end{align} 

Then, there is a universal constant $C>0$ such that for $i=1,2$ and $t \in I_{t_0}$
\begin{equation} \label{eq:s}
  |\partial_t  s_{M^i_v}  (\overline c_i (t))| \le C \max  \|\overline c_i' (t)\| .  
\end{equation}

\begin{lemma} \label{lem:C1glue}
  Let $K \Subset J \Subset I$ be intervals and let $s : J \to G$ be of class $C^1$. 
  Then there is $\tilde s : I \to G$ of class $C^1$ such that 
  \begin{itemize}
  \item[(i)] $s|_{K}= \tilde s|_{K}$. 
  \item[(ii)] $\|s'\|_{L^\infty (K)}  =  \|\tilde s'\|_{L^\infty (I)} $.
  \item[(iii)]  $\tilde s$ is constant on each component of $I\setminus J$.
  \end{itemize}
\end{lemma}

\begin{proof}
 We may extend $s|_K$ through the endpoints of $K = (t_-,t_+)$ using the exponential mapping in the direction 
 $s'(t_\pm)$.   
 More precisely, for the right endpoint $t_+$ set $g = s(t_+) \in G$ and $s'(t_+) = T_e \mu_g.X$ for $X \in \mathfrak g$
 (where $\mu_g(h) = gh$ denotes left translation on $G$), and define 
 $$
 \tilde s(t) = g  \exp(\vh(t-t_+) X), 
 $$
 where $\vh(t) = \int_0^t \ps(u)\, du$ for  
 \begin{align*}
   \ps(t) = 
   \begin{cases}
     1 & t\le 0 \\
     1- \frac{t}{\de} & 0 \le t \le \de \\
     0 & t\ge \de
   \end{cases}
 \end{align*}
 and where $\de$ denotes the distance of the right endpoints of $K$ and $J$.     
\end{proof}

Fix an open interval $K \Subset I_{t_0}$, $t_0\in K$.  By Lemma \ref{lem:C1glue}, 
we may extend each $s_{M^i_v}  (\overline c_i (t))$ to a $C^1$ map  $s_i : I_i\to G$ that coincides with  $s_{M^i_v}  (\overline c_i (t))$ on $K$ and 
is constant in the complement of $I_{t_0}$.  Let us then shrink $I_1$ and $I_2$ so that their union $I_1\cup I_2$ 
does not change but $I_1\cap I_2=K$.  Then we set 
\begin{align*}
\overline c(t) :=
s_{M^i_v}  (\overline c_i (t))^{-1}\ \overline c_i (t), \quad  \text{ if } t\in I_i, \ \ i=1,2,
\end{align*} 
which is well-defined by \eqref{gluing}.  Moreover, 
\begin{equation} \label{eq:cprime}
  \|\overline c'(t)\| \le C \max \{ \|\overline c_1' (t)\|, \|\overline c_2'  (t)\| \}, \quad t \in I_1 \cup I_2,   
\end{equation} 
for a universal constant $C>0$, where we set $\overline c_i' (t):=0$ if $t \not \in I_i$.

\subsection{\texorpdfstring{$\tC{p}{m}$}{pCm}-functions}

Later in the proof  we shall need a result on functions defined near $0 \in\R$ that become $C^m$ when 
multiplied with the monomial $t^p$. 

\begin{definition*}
  Let $p,m \in \N$ with $p \le m$.
A continuous complex valued function $f$ defined near $0 \in \R$ is called a \emph{$\tC{p}{m}$-function} 
if $t \mapsto t^p f(t)$ belongs to $C^m$.  
\end{definition*}

Let $I \subseteq \R$ be an open interval containing $0$. 
Then $f : I \to \C$ is $\tC{p}{m}$ if and only if 
it has the following properties, cf.\ \cite[4.1]{Spallek72}, \cite[Satz~3]{Reichard79}, or \cite[Theorem~4]{Reichard80}:
\begin{itemize}
\item $f \in C^{m-p}(I)$.
\item $f|_{I\setminus \{0\}} \in C^m(I\setminus \{0\})$.
\item $\lim_{t \to 0} t^k f^{(m-p+k)}(t)$ exists as a finite number for all $0 \le k \le p$. 
\end{itemize}

\begin{proposition} \label{prop:pCm}
If $g=(g_1,\ldots,g_n)$ is $\tC{p}{m}$ and $F$ is $C^m$ near $g(0)\in \C^n$, then $F \o g$ is $\tC{p}{m}$.
\end{proposition}

\begin{proof}
Cf.\ \cite[Theorem~9]{Reichard80} or \cite[Proposition~3.2]{Rainerfin}. Clearly $g$ and $F \o g$ are $C^{m-p}$ near $0$ and $C^m$ off $0$. 
By Fa\`a di Bruno's formula \cite{FaadiBruno1855}, for $1 \le k \le p$ and $t\ne0$,
\begin{align*} 
    \frac{t^k (F \o g)^{(m-p+k)}(t)}{(m-p+k)!} &= \sum_{\ell\ge 1}  \sum_{\al \in A}
    \frac{t^{k-|\be|}}{\ell!}  d^\ell F(g(t)) 
    \Big( 
    \frac{t^{\be_1} g^{(\al_1)}(t)}{\al_1!},\dots,
    \frac{t^{\be_\ell} g^{(\al_\ell)}(t)}{\al_\ell!}\Big) \\
    A &:= \{\al\in \N_{>0}^\ell : \al_1+\dots+\al_\ell =m-p+k\} \\
    \be_i &:= \max\{\al_i - m + p,0\}, \quad |\be| = \be_1+\dots+\be_\ell \le k,
\end{align*}
whose limit as $t\to 0$ exists as a finite number by assumption.
\end{proof}

\subsection{End of proof}\label{endofproof}

We distinguish three kinds of points $t_0 \in I$:
\begin{description}
  \item[Case 0\phantomsection\label{Case0}] $c_1(t_0) \ne 0$, or
  \item[Case 1\phantomsection\label{Case1}] $c_1(t_0) = 0$, thus $c_1'(t_0) = 0$ by \eqref{eq:red1}, and $c_1''(t_0) \ne 0$, or
  \item[Case 2\phantomsection\label{Case2}] $c_1(t_0) = c_1'(t_0) = c_1''(t_0)=0$.
\end{description}
Near points of Case~\hyperref[Case0]{0} there are local $C^1$-lifts, by the algorithm in Subsection \ref{ssec:algorithm}.

Let us prove that we also have local $C^1$-lifts near points $t_0$ of Case~\hyperref[Case1]{1}. 
For simplicity of notation let $t_0=0$.
Then  
$c_1(t) \sim t^2$  and hence $c_i(t) = O(t^{d_i})$.
Therefore,   
\[
  \underline c(t) := \big(t^{-2} c_1(t), t^{-d_2 } c_2 (t),\ldots, t^{-d_n } c_n (t) \big) : 
  I_1 \to  \si(V) \subseteq \R^{n}, 
\]
defined on a neighborhood $I_1$ of $0$, is continuous.  
By Lemma~\ref{lem:red} the lifting problem reduces to the curve 
$c^* = (c^*_i)_{i=1}^m$,
\begin{equation} \label{eq:cstar2}
   c^*_i(t) = t^{e_i} \vh_i\big(t^{-2} c_1(t), t^{-d_2 } c_2 (t),\ldots, t^{-d_n } c_n (t)\big),
   \quad e_i = \deg \ta_i,   
\end{equation}  
in the orbit space $\ta(N_v)$ of any slice representation $G_v \acts N_v$ so that $v \in \si^{-1}(\underline c(0))$. 
Then $c^*_i$ is of class $C^{e_i}$ at $0$, by Proposition \ref{prop:pCm}, and of class $C^d$ in the complement of $0$. 
After removing fixed points of $G_v \acts N_v$, 
we may assume that the curve 
\begin{equation*}
   \underline c^*(t) := \big(t^{-e_1} c^*_1(t), t^{-e_2} c^*_2(t),\ldots, t^{-e_m} c^*_m (t)\big)
\end{equation*} 
in $\ta(N_v)$ vanishes at $t=0$, since $\underline c(0) = \si(v)$ (cf.\ \eqref{eq:vh}). 
Thus $c^*_i(t) = o(t^{e_i})$, for all $i$. 

\begin{lemma} \label{lem:Case1}
In this situation,
for any $\varepsilon>0$ there is a neighborhood $I_\varepsilon$ of $0$ in $I$ such that 
for every  $t_0\in I_\varepsilon \setminus \{0\}$ the assumptions \eqref{A.2}--\eqref{A.4} are satisfied 
for the reduced curve $c^*$ from \eqref{eq:cstar2} with $A\le \varepsilon$.  
\end{lemma}

\begin{proof}
Here we have to deal with the fact that 
$c^*$ is not necessarily of class $C^e$. 
Let $I_0 = (-\delta, \delta)$ and $I_1 = (-2\delta, 2\delta)$.  Since $(c^*_1)''(0) = 0 $ and $c^*_1(t)$ is 
of class $C^2$, the constant $A_1$ of \eqref{A1A2} for $c^*$ can be chosen arbitrarily small.  
This is what we need to get \eqref{A.2}--\eqref{A.3} with arbitrarily small $A$. 
 
We have $c^*_i \in C^{e_i}$ near $0$ (and $c^*_i \in C^{d}$ off $0$) and $(c^*_i)^{(k)}(0) = 0$ for all $k \le e_i$.  
Therefore  for an arbitrary $A>0$ there is a neighborhood $I_1$ in which \eqref{A.4} holds 
for all $i$ and $k=e_i$,  
and then, by Lemma \ref{shortcut}, in a smaller neighborhood, for all $i$ and all  
$k\le e_i$.

Finally, given $A>0$ we show \eqref{A.4} for $k> e_i$ and $\delta$ sufficiently small.
Let $\hat A$ denote the constant $A$ for which \eqref{A.2}--\eqref{A.4} holds for $c$.
By \eqref{derivatives2}, for some constant $C = C(G \acts V)$, 
  \[
    |{(c^*_i)}^{(k)}(t)| \le C \hat A^k\, |c_1(t)|^{\frac{e_i-k}{2}}
    \le C \hat A^k\, \ps (t)  |c^*_1 (t) |^{\frac{e_i-k}{2}},
  \]
which gives the required result since  $\ps (t) =  | c^*_1 (t) / c_1(t) |^{\frac{k-e_i}{2}} = o(1)$ for $k>e_i$.
\end{proof}

By induction, we may conclude from Lemma \ref{lem:Case1} that there is a $C^1$-lift near $0$. 

We may now glue the local lifts, according to Subsection \ref{ssec:glue}.  
Let $J$ be a connected component of  the complement $ I'$ 
of the flat points (i.e., the points in Case~\hyperref[Case2]{2}).  
Then there exists an open cover  $\mathcal J= \{J_i\}_{i\in \Z}$  of $J$, with $C^1$-lifts $\overline c_i$ of $c|_{J_i}$, and  such that  $J_i\cap J_j \ne \emptyset$ if and only if $|i-j|\le 1$.  
By Subsection \ref{ssec:glue} we may assume that there are $C^1$-maps $ s_{i, \pm} : J_i \to G$ such that 
on $J_i\cap J_{i+1}$ 
\begin{align}\label{gluing2}
 s_{i, +} (t)  \ \overline c_i(t)  =  s_{i+1,-}(t)  \  \overline c_{i+1}(t) .
\end{align} 
Moreover, by Lemma \ref{lem:C1glue}, we may assume that there is $t_i\in J_i \setminus (J_{i-1}\cup J_{i+1})$ such that both 
$ s_{i, \pm}$ are constant, say equal $ g_{i, \pm}$, in a neighborhood $J_{t_i}$ of $t_i$.   Thus we may glue 
$g_{i,-}^{-1} s_{i, -}$ and $g_{i,+}^{-1} s_{i, + }$ into a single map  $s_{i} : J_i \to G$ that equals 
$g_{i,-}^{-1} s_{i, -}$ 
for $t\le t_i$ and  $g_{i,+}^{-1} s_{i, + }$ for $t\ge t_i$.   Then 
\begin{align}\label{gluing3}
g_{i,+} s_{i} (t)  \ \overline c_i(t)  = g_{i+1,-} s_{i+1}(t)  \  \overline c_{i+1}(t) .
\end{align}

\begin{lemma}\label{coboundarylemma}
There are $h_i\in G$ such that 
\begin{align}\label{gluing4}
h_{i} s_{i} (t)  \ \overline c_i(t)  = h_{i+1} s_{i+1}(t)  \  \overline c_{i+1}(t) .
\end{align} 
\end{lemma}

\begin{proof}
In view of \eqref{gluing3} it suffices to find $h_i$ such that $g_{i+1,-}^{-1}  g_{i,+} = h_{i+1}^{-1}  h_{i}$.  
So we may fix $h_0=e$ and then define them  inductively by $h_{i+1} = h_i g_{i,+}^{-1}  g_{i+1,-}$. 

(Note that the existence of such $h_i$ simply means that the cocycle $g_{i+1,-}^{-1}  g_{i,+}$ is 
a \v Cech coboundary, that  is   clear because 
  $\check H^1 (\mathcal J; G)=0$.)
\end{proof}

In this way we obtain a $C^1$-lift $\overline c$ of $c$ restricted to $I'$ with the property that $\|\overline c'(t)\|$ 
is dominated (up to a universal constant) by $A_0$ defined by \eqref{formulaforA}, 
thanks to \eqref{eq:cprime}. 
The lift $\overline c$ extends trivially to flat points $t_0$ from Case~\hyperref[Case2]{2}. 
At each such point $t_0$, $\overline c$ is differentiable with $\overline c'(t_0) = 0$. 
It remains to check that $\overline c'(t) \to 0$ as $t\to t_0$. 
This is a consequence of the following lemma, where without loss of generality $t_0=0$.

\begin{lemma}\label{flatcaselemma}
If $c_1(0) = c_1'(0) = c_1''(0)=0$,
then for any $\varepsilon>0$ there is $\delta >0$ such that for $I_0 = (-\delta, \delta)$, $I_1 = (-2\delta, 2\delta)$,  
and $A_0$ defined by \eqref{formulaforA} we have $A_0 \le \varepsilon$.  
\end{lemma}

\begin{proof}
This follows immediately from the formulas \eqref{formulaforA} and \eqref{A1A2}.  
\end{proof} 

The proof of Theorem~\ref{mainC1} is complete.

\section{Real analytic lifts} 

It was shown in \cite{AKLM00} that a real analytic curve $c \in C^\om(I,\si(V))$ admits local real analytic lifts near 
every point $t_0 \in I$, and that the local lifts can be glued to a global real analytic lift if $G \acts V$ is polar. 
We will now show that real analytic gluing is always possible.  

\begin{theorem} \label{thm:realanalytic}
  Let $\rep$ be a real finite dimensional orthogonal representation of a compact Lie group. 
  Then any $c \in C^{\om}(I,\si(V))$ admits a lift $\overline c \in C^{\om}(I,V)$.
\end{theorem} 

\begin{proof}
The local lifts can be glued thanks to the fact that  
 \begin{align}\label{cohomology}
 \check H^1 (I, G^a) = 0,
 \end{align}
 where $G^a$ denotes the sheaf of real analytic maps $I\supseteq U \to G$.  
 This is a deep result, suggested by Cartan in \cite{Cartan56}, \cite{Cartan60}, and proven by Tognoli \cite{Tognoli69}.   
 
 Indeed, let $\mathcal I = \{I_i\}$ be a locally finite cover of $I$ with real analytic lifts 
 $\overline c_i$ of $c|_{I_i}$ (which exist by the result of \cite{AKLM00}).  Then, by Lemma 3.8 of  \cite{AKLM00}, we may assume that 
 if $I_i\cap I_j \ne \emptyset$ 
 then there is real analytic $s_{ij} : I_i\cap I_j  \to G$ such that on $I_i\cap I_j$ 
 $$
s_{ij} \overline c_i =\overline c_j .
 $$
 By \eqref{cohomology}, after replacing $\mathcal I$ by its refinement if necessary,  there are real analytic 
 $h_i: I_i\to G$ such that $s_{ij}= h_j^{-1} h_i$ on $I_i\cap I_j$ and then 
 $$
 \overline c (t) = h_i (t) \overline c_i (t)  , \   \text { if }  \  t\in I_i,
 $$
defines a global lift.  
\end{proof}

\def\cprime{$'$}
\providecommand{\bysame}{\leavevmode\hbox to3em{\hrulefill}\thinspace}
\providecommand{\MR}{\relax\ifhmode\unskip\space\fi MR }
\providecommand{\MRhref}[2]{%
  \href{http://www.ams.org/mathscinet-getitem?mr=#1}{#2}
}
\providecommand{\href}[2]{#2}


\end{document}